\documentclass[12pt,a4paper,reqno]{amsart}

\usepackage[left=2.8cm,right=2.8cm,top=2.8cm,bottom=2.8cm]{geometry}

\usepackage{amsmath,amsfonts,amsthm,amsopn,amssymb,cite,mathrsfs}
\usepackage{cancel,marginnote}

\usepackage{color,enumitem,graphicx}
\usepackage[colorlinks=true,urlcolor=blue,
citecolor=red,linkcolor=blue,linktocpage,pdfpagelabels]{hyperref}
\usepackage[english]{babel}

\usepackage[hyperpageref]{backref}

\makeindex

\numberwithin{equation}{section}
\newtheorem{theorem}{Theorem}[section]
\newtheorem{lemma}{Lemma}[section]

\newtheorem{remark}{Remark}[section]

\newtheorem{definition}{Definition}[section]

\title[Schr\"{o}dinger equation with local and nonlocal operator]
{The Best constant in the Gagliardo-Nirenberg inequality for the mixed local and nonlocal Laplacian}


\begin{document}
\maketitle

\centerline{\scshape Hichem Hajaiej*}
\medskip
{\footnotesize
\centerline{Department of Mathematics, California State University,}
\centerline{Los Angeles, California, USA}
\centerline{hhajaie@calstatela.edu}
}

\medskip

\centerline{\scshape Yu Su}
\medskip
{\footnotesize
\centerline{School of Mathematics and Big Data, Anhui University of Science and Technology,}
\centerline{Huainan, Anhui 232001, China}
\centerline{yusumath@aust.edu.cn}
}

\begin{abstract}
In this paper, we establish the best constant of the Gagliardo-Nirenberg inequality for the mixed local and nonlocal Laplacian.
In the breakthrough paper of Weinstein \cite{Weinstein1983CMP},
he proved the best constant by establishing the relationship between the optimizers of the best constant and the ground state solutions of the underlying equation.
To achieve this goal, the Poho\v{z}aev identity plays a key role.
In our problem, due to the appearance of a local and a nonlocal operator and the loss of the $L^{2}$ term,
it is hard to get the regularity result.
It doesn't seem reasonable to think that all the weak solutions of equation satisfy Poho\v{z}aev identity.
%Hence, it is hard to establish the relationship between the optimizer of the best constant and the ground state solution.
We overcome this difficulty by using the Nehari-Poho\v{z}aev manifold method.
We not only prove that the optimizer of the best constant is the ground state solution of the equation, but we also establish the best constant.
\end{abstract}

\section{Introduction}
In this paper, we establish the best constant of the Gagliardo-Nirenberg inequality involving the Laplacian and the fractional Laplacian, that is, for $N\geqslant3$, $s\in(0,1)$ and $p\in (2^{*}_{s},2^{*})$, we show that there is a positive constant $C_{N,p,s}>0$ such that for all $u\in E$,
\begin{equation}\label{GN}
\begin{aligned}
\int_{\mathbb{R}^{N}}
|u|^{p}
\mathrm{d}x
\leqslant
C_{N,p,s}
\|u\|_{D^{s,2}(\mathbb{R}^{N})}^{\frac{2N-p(N-2)}{2(1-s)}}
\|u\|_{D^{1,2}(\mathbb{R}^{N})}^{\frac{p(N-2s)-2N}{2(1-s)}},
\end{aligned}
\tag{$GN$}
\end{equation}
where $2^{*}=\frac{2N}{N-2}$ is the critical Sobolev exponent, and $2_{s}^{*}=\frac{2N}{N-2s}$ is the  fractional critical Sobolev exponent, and $E$ is a mixed Sobolev space  defined in Section 2.

To state the best constant $C_{N,p,s}$ precisely, we consider the equation
\begin{equation}\label{P}
\begin{aligned}
-\Delta u
+(-\Delta)^{s}u =|u|^{p-2}u, \ \ x\in\mathbb{R}^{N},
\end{aligned}
\tag{$P$}
\end{equation}
where $N\geqslant3$, $s\in(0,1)$, $p\in (2^{*}_{s},2^{*})$,
and $(-\Delta)^{s}$ is the so-called fractional
Laplacian, see \cite{Caffarelli-Silvestre2007PDE}.
For $p=2$ and on a bounded domain in $\mathbb{R}^{N}$,
equation \eqref{P} arises in the population dynamics model with both classical and nonlocal diffusion introduced by
Dipierro-Lippi-Valdinocci \cite{Dipierro-Lippi-Valdinocci2022AIHP}.
Moreover, Biagi-Dipierro-Valdinoci-Vecchi \cite{Biagi-Dipierro-Valdinoci-Vecchi2021CPDE} applied it to study the different types of ``regional" or ``global" restrictions that may reduce the spreading of a pandemic disease.
Dipierro-Valdinocci \cite{Dipierro-Valdinocci2021PA} introduced the description of an ecological niche for a mixed local and nonlocal dispersal.
Cluni-Gusella-Mugnai-Lippi-Pucci
\cite{Cluni-Gusella-Mugnai-ProiettiLippi-Pucci2023ME}
proposed a model describing the nonlocal behavior of an elastic
body using a peridynamical approach.
Hence, mixed operators have recently received a great attention due to their numerous and important applications. Many aspects were addressed including regularity theory
\cite{Biagi-Dipierro-Valdinoci-Vecchi2021CPDE,Su-Valdinoci-Wei-Zhang2022MZ,DeFilippis-Mingione2024MA,Byun-Kumar-Lee2024CVPDE};
existence and non-existence
results
\cite{Biagi-Mugnai-Vecchi2022CCM,Chergui-Gou-Hajaiej2023CVPDE,Salort-Vecchi2022DIE,LuoTJ-Hajaiej2022ANS,Maione-Mugnai-Vecchi2023FCAA};
eigenvalue problems \cite{Dipierro-Proietti Lippi-Valdinoci2022AA,Del Pezzo-Ferreira-Rossi2019FCAA,ChenHY-Bhakta-Hajaiej2022JDE};
shape optimization and calculus of
variations \cite{Biagi-Dipierro-Valdinoci-Vecchi2023JAM}. New features arise for the mixed local and nonlocal operators, see \cite{Su-Hajaiej-Shi2026, Su-Hajaiej2026} for more details.

In this paper, we study the best constant $C_{N,p,s}$ of inequality \eqref{GN} of the mixed operators type equation \eqref{P}.

\subsection{Gagliardo-Nirenberg inequalities: Historical background}

$~$\\
 In 1914, Landau \cite{Landau1914PLMS} showed the inequality
\begin{equation*}
\begin{aligned}
\|u'\|_{L^{\infty}(\mathbb{R})}
\leqslant
C
\|u''\|_{L^{\infty}(\mathbb{R})}^{\frac{1}{2}}
\|u\|_{L^{\infty}(\mathbb{R})}^{\frac{1}{2}}.
\end{aligned}
\end{equation*}
In 1949, Kolmogoroff
\cite{Kolmogoroff1949} extended it for higher-order derivatives.
In 1957, Stein
\cite{Stein1957Annals}
improved Landau's result  for $p\in[1,\infty)$ as
\begin{equation*}
\begin{aligned}
\|u'\|_{L^{p}(\mathbb{R})}
\leqslant
C
\|u''\|_{L^{p}(\mathbb{R})}^{\frac{1}{2}}
\|u\|_{L^{p}(\mathbb{R})}^{\frac{1}{2}}.
\end{aligned}
\end{equation*}
In 1958, Nash \cite{Nash1958AJM} and Ladyzenskaja \cite{Ladyzenskaja1958} independently proved, for $N=2$,
\begin{equation*}
\begin{aligned}
\|u\|_{L^{4}(\mathbb{R}^{N})}
\leqslant
C
\|u\|_{D^{1,2}(\mathbb{R}^{N})}^{\frac{1}{2}}
\|u\|_{L^{2}(\mathbb{R}^{N})}^{\frac{1}{2}}.
\end{aligned}
\end{equation*}
In 1959, Gagliardo \cite{Gagliardo1959RDM} and Nirenberg \cite{Nirenberg1959Pisa} independently showed the Gagliardo-Nirenberg inequality
\begin{equation*}
\begin{aligned}
\|u\|_{D^{j,p}(\mathbb{R}^{N})}
\leqslant
C
\|u\|_{D^{k,r}(\mathbb{R}^{N})}^{\theta}
\|u\|_{L^{q}(\mathbb{R}^{N})}^{1-\theta},
\end{aligned}
\end{equation*}
where
\begin{equation*}
\begin{aligned}
\begin{cases}
0\leqslant j<k,~~j,k\in \mathbb{N},\\
\theta=\frac{j}{k},\\
\frac{k}{p}=\frac{j}{r}+\frac{k-j}{q}.
\end{cases}
\end{aligned}
\end{equation*}
In 2018, Brézis-Mironescu \cite{Brezis-Mironescu2018Poincare} improved the Gagliardo-Nirenberg's result for $j,k\geqslant0$.
For more related works, we refer the reader to
\cite{Brezis-Mironescu2019JFA,Fila-Winkler2019AM,DelPino-Dolbeault2002JMPA,CorderoErausquin-Nazaret-Villani2004AM,Agueh2008NODEA,Hajaiej-Molinet-Ozawa-Wang2011}.

\subsection{Best constant of the Gagliardo-Nirenberg inequality for the Laplacian}

$~$\\
If we replace $\|u\|_{D^{s,2}(\mathbb{R}^{N})}$ by $\|u\|_{L^{2}(\mathbb{R}^{N})}$, then inequality \eqref{GN} reads, for all $u\in H^{1}(\mathbb{R}^{N})$ and $C_{1}>0$ is the best constant:
\begin{equation}\label{1.1}
\begin{aligned}
\|u\|_{L^{p}(\mathbb{R}^{N})}^{p}
\leqslant
C_{1}
\|u\|_{D^{1,2}(\mathbb{R}^{N})}^{\frac{N(p-2)}{2}}
\|u\|_{L^{2}(\mathbb{R}^{N})}^{p-\frac{N(p-2)}{2}}.
\end{aligned}
\end{equation}
For $N=1$, $p\in(2,2^{*})$ and $2^{*}=\infty$, Nagy \cite{Nagy1941} established inequality \eqref{1.1} and the best constant
\begin{equation*}
\begin{aligned}
C_{1}=
\left(\frac{2}{p}\|\psi\|_{L^{2}(\mathbb{R})}^{p-2}\right)^{-\frac{1}{p}},
\end{aligned}
\end{equation*}
where
$\psi(x)=\left(\frac{p}{2}\right)^{\frac{1}{p-2}}
\frac{1}{[\cosh(\frac{p-2}{2}x)]^{\frac{2}{p-2}}}$ is the unique positive even solution in $C^{2}(\mathbb{R})$ of the equation, see \cite[Proposition 5.13]{Frank-Lenzmann2013Acta},
\begin{equation}\label{S1}
\begin{aligned}
-\Delta u+u =|u|^{p-2}u, \ \ x\in\mathbb{R}.
\end{aligned}
\tag{$S_{1}$}
\end{equation}

For $N\geqslant2$, $p\in(2,2^{*})$ and $2^{*}=\frac{2N}{N-2}$,
Weinstein \cite[Theorem B]{Weinstein1983CMP} established the best constant $C_{1}$. His method consists of four steps:

\begin{enumerate}
\item [Step 1.]
By introducing Weinstein functional, using symmetric decreasing rearrangement and compactness Lemma, he showd the existence of optimal functions for inequality \eqref{1.1};

\item [Step 2.]
From the minimizing problem of Weinstein functional to a Euler-Lagrange equation, he proved that there is a optimal function $u\in H^{1}(\mathbb{R}^{N})$ which is also a weak solution of the following equation
\begin{equation}\label{SN}
\begin{aligned}
-\Delta u+u =|u|^{p-2}u, \ \ x\in\mathbb{R}^{N}.
\end{aligned}
\tag{$S_{N}$}
\end{equation}

\item [Step 3.]
By applying the uniqueness of decaying positive solutions of equation \eqref{SN} \cite{Serrin-M.X.Tang2000IUMJ} and Poho\v{z}aev identity of equation \eqref{SN} \cite{Berestycki-Lions1983ARMA}, he investigated that there is a radial positive optimal function which is also a radial positive ground state solution of equation \eqref{SN}.

\item [Step 4.]
Finally,
he concluded the best constant as
\begin{equation*}
\begin{aligned}
C_{1}=
\left(\frac{2}{p}\|\psi\|_{L^{2}(\mathbb{R}^{N})}^{p-2}\right)^{-\frac{1}{p}}
=c^{-\frac{1}{p}},
\end{aligned}
\end{equation*}
where $\psi$ is the unique radial positive ground state solution of equation \eqref{SN}, and $c=\frac{2}{p}\|\psi\|_{L^{2}(\mathbb{R}^{N})}^{p-2}$ is the ground state energy of equation \eqref{SN} on $H^{1}(\mathbb{R}^{N})$.
\end{enumerate}

\subsection{Best constant of the Gagliardo-Nirenberg inequality for the fractional Laplacian}

$~$\\
If one replaces $\|u\|_{D^{1,2}(\mathbb{R}^{N})}$ by $\|u\|_{L^{2}(\mathbb{R}^{N})}$, then inequality \eqref{GN} reads, for all $u\in H^{s}(\mathbb{R}^{N})$ and $C_{2}>0$ is the best constant:
\begin{equation}\label{1.2}
\begin{aligned}
\|u\|_{L^{p}(\mathbb{R}^{N})}^{p}
\leqslant
C_{2}
\|u\|_{D^{s,2}(\mathbb{R}^{N})}^{\frac{N(p-2)}{2s}}
\|u\|_{L^{2}(\mathbb{R}^{N})}^{p-\frac{N(p-2)}{2s}}.
\end{aligned}
\end{equation}
\eqref{1.2} was first established in \cite{Hajaiej-Molinet-Ozawa-Wang2011}. For $N=1$, $s=\frac{1}{2}$ and $p=4$, Amick-Toland \cite{Amick-Toland1991Acta} considered the following equation
\begin{equation}\label{FS1}
\begin{aligned}
(-\Delta)^{\frac{1}{2}} u+u =|u|^{2}u, \ \ x\in\mathbb{R}.
\end{aligned}
\tag{$FS_{1,4}$}
\end{equation}
They showed that the unique positive ground state solution of equation \eqref{FS1} is
\begin{equation*}
\begin{aligned}
\psi(x)=\frac{1}{1+x^{2}}\in H^{\frac{1}{2}}(\mathbb{R}).
\end{aligned}
\end{equation*}
Then mimicking the method of Weinstein, one can show the optimizers and determine the best constant for inequality \eqref{1.2} directly by the uniqueness of the ground states, the rearrangement inequalities and Poho\v{z}aev identity.

Frank-Lenzemann \cite[Theorem 2.4]{Frank-Lenzmann2013Acta} established the non-degeneracy and uniqueness of positive ground state solution of
\begin{equation}\label{FS2}
\begin{aligned}
(-\Delta)^{s} u+u =|u|^{p-2}u, \ \ x\in\mathbb{R},
\end{aligned}
\tag{$FS_{1,p}$}
\end{equation}
where
\begin{equation*}
\begin{aligned}
\begin{cases}
N=1,~~s\in(0,\frac{1}{2}),&p\in(2,2_{s}^{*}),~~\mathrm{where}~~2_{s}^{*}=\frac{4s}{1-2s},\\
N=1,~~s\in[\frac{1}{2},1),&p\in(2,2_{s}^{*}),~~\mathrm{where}~~2_{s}^{*}=+\infty.
\end{cases}
\end{aligned}
\end{equation*}
By using the uniqueness of the ground states, Poho\v{z}aev identity and the method of Weinstein, one can show that the optimizers and the best constant for inequality \eqref{1.2}.
Moreover, they \cite[Lemma A.5]{Frank-Lenzmann2013Acta} obtained the uniform bound for the best constant of inequality \eqref{1.2} under
\begin{equation*}
\begin{aligned}
\begin{cases}
N=1,~~s\in[\frac{p-2}{2p},\frac{1}{2}),&p\in(2,2_{s}^{*}),~~\mathrm{where}~~2_{s}^{*}=\frac{4s}{1-2s},\\
N=1,~~s\in[\frac{1}{2},1),&p\in(2,2_{s}^{*}),~~\mathrm{where}~~2_{s}^{*}=+\infty.
\end{cases}
\end{aligned}
\end{equation*}
For $N\geqslant 2$, $s\in(0,1)$, $p\in(2,2_{s}^{*})$ and $2_{s}^{*}=\frac{2N}{N-2}$,
Frank-Lenzmann-Silvestre
\cite{Frank-Lenzmann-Silvestre2016CPAM} considered the non-degeneracy and the uniqueness of positive ground state solutions of
\begin{equation}\label{FS3}
\begin{aligned}
(-\Delta)^{s} u+u =|u|^{p-2}u, \ \ x\in\mathbb{R}^{N}.
\end{aligned}
\tag{$FS_{N,p}$}
\end{equation}
Similar to the proof of Weinstein and using the uniqueness of the ground state and the Poho\v{z}aev identity, one deduces the optimizers and the best constant for inequality \eqref{1.2}.

\subsection{Best constant of the Gagliardo-Nirenberg inequality with $L^{p}$ norm}

$~$\\
If one replaces $\|u\|_{D^{s,2}(\mathbb{R}^{N})}$ by $\|u\|_{L^{\frac{p}{2}+1}(\mathbb{R}^{N})}$, then inequality \eqref{GN} is, for all $u\in D^{1,2}(\mathbb{R}^{N})\cap L^{\frac{p}{2}+1}(\mathbb{R}^{N})\cap L^{p}(\mathbb{R}^{N})$ and $C_{3}>0$ is the best constant,
\begin{equation}\label{1.3}
\begin{aligned}
\|u\|_{L^{p}(\mathbb{R}^{N})}
\leqslant
C_{3}
\|u\|_{D^{1,2}(\mathbb{R}^{N})}^{\theta}
\|u\|_{L^{\frac{p}{2}+1}(\mathbb{R}^{N})}^{1-\theta},
\end{aligned}
\end{equation}
where
\begin{equation*}
\begin{aligned}
\theta=\frac{N(\frac{p}{2}-1)}{\frac{p}{2}(N+2-(N-2)\frac{p}{2})},~~p\in(2,2^{*}].
\end{aligned}
\end{equation*}
The optimizer of inequality \eqref{1.3} is related to the ground state solution of the equation
\begin{equation*}
\begin{aligned}
-\Delta u+|u|^{\frac{p}{2}-1}u
=|u|^{p-2}u, \ \ x\in\mathbb{R}^{N}.
\end{aligned}
\end{equation*}
By using minimizing method, the uniqueness of the ground state and the Poho\v{z}aev identity,
Del Pino-Dolbeault
\cite[Theorem 1]{DelPino-Dolbeault2002JMPA} proved the best constant and the optimizer of inequality \eqref{1.3} as
\begin{equation*}
\begin{aligned}
C_{3}
=
\left(
\frac{Y(\frac{p}{2}-1)^{2}}{2\pi N}
\right)^{\frac{\theta}{2}}
\left(
\frac{2Y-N}{2Y}
\right)^{\frac{1}{p}}
\left(
\frac{\Gamma(Y)}{\Gamma(Y-\frac{N}{2})}
\right)^{\frac{\theta}{N}},
~~
Y=\frac{p+2}{p-2},
\end{aligned}
\end{equation*}
and
\begin{equation*}
\begin{aligned}
\psi(x)=\left(\frac{1}{\sigma^{2}+|x-\bar{x}|^{2}}\right)^{\frac{2}{p-2}},~~\sigma>0,~~\bar{x}\in \mathbb{R}^{N}.
\end{aligned}
\end{equation*}
Moreover, Del Pino-Dolbeault\cite[Theorem 2]{DelPino-Dolbeault2002JMPA}
considered the following Gagliardo-Nirenberg inequality
\begin{equation}\label{1.4}
\begin{aligned}
\|u\|_{L^{\frac{p}{2}+1}(\mathbb{R}^{N})}
\leqslant
C_{4}
\|u\|_{D^{1,2}(\mathbb{R}^{N})}^{\theta}
\|u\|_{L^{p}(\mathbb{R}^{N})}^{1-\theta},
\end{aligned}
\end{equation}
where $u\in D^{1,2}(\mathbb{R}^{N})\cap L^{\frac{p}{2}+1}(\mathbb{R}^{N})\cap L^{p}(\mathbb{R}^{N})$ and $C_{4}>0$ is the best constant, and
\begin{equation*}
\begin{aligned}
\theta=\frac{N(1-\frac{p}{2})}{(1+\frac{p}{2})(N-(N-2)\frac{p}{2})},~~p\in(1,2).
\end{aligned}
\end{equation*}
The best constant and optimizer of inequality \eqref{1.4} are
\begin{equation*}
\begin{aligned}
C_{4}
=
\left(
\frac{Y(\frac{p}{2}-1)}{2\pi N}
\right)^{\frac{\theta}{2}}
\left(
\frac{2Y}{2Y+N}
\right)^{\frac{1-\theta}{p}}
\left(
\frac{\Gamma(Y+1+\frac{N}{2})}{\Gamma(Y+1)}
\right)^{\frac{\theta}{N}},
~~
Y=\frac{p+2}{2-p},
\end{aligned}
\end{equation*}
and
\begin{equation*}
\begin{aligned}
\psi(x)=\left(\sigma^{2}-|x-\bar{x}|^{2}\right)_{+}^{\frac{2}{2-p}},~~\sigma>0,~~\bar{x}\in \mathbb{R}^{N}.
\end{aligned}
\end{equation*}
By using the optimal transport techniques, Cordero Erausquin-Nazaret-Villani
\cite{CorderoErausquin-Nazaret-Villani2004AM} extended the results of \cite{DelPino-Dolbeault2002JMPA} from the Laplacian to the $p$-Laplacian by using an innovative method.
Moreover, combining the minimizing method, uniqueness of the ground state , Poho\v{z}aev identity and the optimal transport techniques, Agueh \cite{Agueh2008NODEA}
extended the results of \cite{CorderoErausquin-Nazaret-Villani2004AM}. For more related works, we refer to \cite{Fila-Winkler2019AM}.

\subsection{Our main results}

$~$\\
Based on the results in subsections 1.2-1.4, we know that the uniqueness of the ground state and the Poho\v{z}aev identity play key roles in the proof of the best constant.
In our case, we are facing the following challenges:

\begin{enumerate}
\item [(i)]
In
\cite{Dipierro-SuXF-Valdinoci-ZhangJW2025DCDS,Su-Valdinoci-Wei-Zhang2022MZ,Su-Valdinoci-Wei-Zhang2025JDE},
the authors showed the regularity results for  mixed operators type equation.
However, due to the loss of $L^{2}$ term in equation \eqref{P},
it is hard to apply their method to get the regularity result.
Without that piece of information, we cannot tell whether the weak solutions of equation \eqref{P} satisfy the Poho\v{z}aev identy.

\item [(ii)]
For equation \eqref{P},
it should be remarked that the delicate issue of uniqueness of solutions (modulo symmetries) as well as the non-degeneracy of the associated linearized operator are two open challenging questions.

\item [(iii)]
It is hard to establish the relationship between the optimizer of the best constant
and the ground state solution of equation \eqref{P}.
\end{enumerate}

We will overcome these challenges thanks to an innovative approach that we will describe later. Let us first state our first result:
%From the above, there is a question:

%\centerline{\bf What about the best constant for inequality \eqref{GN}?}

%The motivation of this paper is to answer this question.
%Firstly, we study the existence of optimizer for inequality \eqref{GN}.
\begin{theorem}\label{Theorem1.1}
Let $N\geqslant3$, $s\in(0,1)$ and $p\in (2^{*}_{s},2^{*})$.
Then we have the following results.

\begin{enumerate}
\item [(1)]
For all $u\in E$, we have
\begin{equation*}
\begin{aligned}
\int_{\mathbb{R}^{N}}
|u|^{p}
\mathrm{d}x
\leqslant
C_{N,p,s}
\|u\|_{D^{s,2}(\mathbb{R}^{N})}^{\frac{2N-p(N-2)}{2(1-s)}}
\|u\|_{D^{1,2}(\mathbb{R}^{N})}^{\frac{p(N-2s)-2N}{2(1-s)}},
\end{aligned}
\end{equation*}
where $C_{N,p,s}>0$ (independent of $u$) is the best constant.

\item [(2)]
There exists a radial non-negative optimizer $u\in E$ for inequality \eqref{GN}.

\item [(3)]
Set $Q=\lambda_{1}u(\lambda_{2}x)$, where
\begin{equation*}
\begin{aligned}
\begin{cases}
\lambda_{1}=
\left[
\frac{p(N-2s)-2N}{2N-p(N-2)}
\right]^{\frac{1}{(1-s)(p-2)}}
\left[
\frac{2p(1-s)}{p(N-2s)-2N}
\frac{1}{\|u\|_{L^{p}(\mathbb{R}^{N})}^{p}}
\right]^{\frac{1}{p-2}},\\
\lambda_{2}=
\left[
\frac{p(N-2s)-2N}{2N-p(N-2)}
\right]^{\frac{1}{2-2s}}.
\end{cases}
\end{aligned}
\end{equation*}
Then $Q$ is a radial non-negative weak solution of equation \eqref{P}, and it is also an optimizer for inequality \eqref{GN}. Moreover, we have
\begin{equation}\label{1.5}
\begin{aligned}
\|Q\|_{D^{1,2}(\mathbb{R}^{N})}^{2}
=&
\frac{p(N-2s)-2N}{2N-p(N-2)}
\|Q\|_{D^{s,2}(\mathbb{R}^{N})}^{2}\\
=&
\frac{p(N-2s)-2N}{2p(1-s)}
\|Q\|_{L^{p}(\mathbb{R}^{N})}^{p},
\end{aligned}
\end{equation}

\begin{equation}\label{1.6}
\begin{aligned}
\frac{N-2}{2}
\|Q\|^{2}_{D^{1,2}(\mathbb{R}^{N})}
+
\frac{N-2s}{2}
\|Q\|^{2}_{D^{s,2}(\mathbb{R}^{N})}
=\frac{N}{p}\|Q\|_{L^{p}(\mathbb{R}^{N})}^{p},
\end{aligned}
\end{equation}
and
\begin{equation}\label{1.7}
\begin{aligned}
C_{N,p,s}^{-1}=
\left(
\frac{2N-p(N-2)}{2p(1-s)}
\right)
^{\frac{2N-p(N-2)}{4(1-s)}}
\left(
\frac{p(N-2s)-2N}{2p(1-s)}
\right)
^{\frac{p(N-2s)-2N}{4(1-s)}}
\|Q\|_{L^{p}(\mathbb{R}^{N})}^{\frac{p(p-2)}{2}}.
\end{aligned}
\end{equation}
\end{enumerate}
\end{theorem}
%\begin{remark}
%In \eqref{1.7}, we present the best constant,
%but $\|Q\|_{L^{p}(\mathbb{R}^{N})}$ is uncertain.
%Based on the results in subsections 1.2-1.4, we know that the ground state uniqueness and Poho\v{z}aev identity play key roles in the proof of best constant.
%Hence, we have the following problems:

%\begin{enumerate}
%\item [(i)]
%In Theorem \ref{Theorem1.1},
%we just show that $Q$ is a weak solution for equation \eqref{P}. But, we do not know that $Q$ is a ground state solution for equation \eqref{P} or not.

%\item [(ii)]
%For equation \eqref{P},
%it should be remarked that the delicate issue of uniqueness of $Q$ (modulo symmetries) as well as the non-degeneracy of the associated linearized operator are open questions.

%\item [(iii)]
%In
%\cite{Dipierro-SuXF-Valdinoci-ZhangJW2025DCDS,Su-Valdinoci-Wei-Zhang2022MZ,Su-Valdinoci-Wei-Zhang2025JDE},
%the authors showed the regularity results for  mixed operators type equation.
%Due to the loss of $L^{2}$ term in equation \eqref{P},
%it is hard to apply their method to get the regularity result.
%By using \eqref{1.5} without any regularity,
%we prove that $Q$ satisfies the Poho\v{z}aev identity \eqref{1.6}.
%However, we do not know that all weak solutions of equation \eqref{P} satisfy Poho\v{z}aev identy or not.
%\end{enumerate}
%\end{remark}

Secondly,
we prove the existence of ground state solution of equation \eqref{P}.
\begin{theorem}\label{Theorem1.2}
Let $N\geqslant3$, $s\in(0,1)$ and $p\in (2^{*}_{s},2^{*})$. Then equation \eqref{P} has a radial ground state solution $\Phi\in E$, which satisfies
\begin{equation*}
\begin{aligned}
I(\Phi)=\inf_{u\in \mathcal{N}}I(u)=c,
\end{aligned}
\end{equation*}
where $I$ is the energy functional of equation \eqref{P} defined as follows
\begin{equation*}
\begin{aligned}
I(u)=
\frac{1}{2}
\left(
\int_{\mathbb{R}^{N}}
|\nabla u|^{2}
\mathrm{d}x
+
\int_{\mathbb{R}^{N}}\int_{\mathbb{R}^{N}}\frac{|u(x)-u(y)|^{2}}{|x-y|^{N+2s}}
\mathrm{d}x\mathrm{d}y\right)
-
\frac{1}{p}
\int_{\mathbb{R}^{N}}
|u|^{p}
\mathrm{d}x,
\end{aligned}
\end{equation*}
and $\mathcal{N}:=\{u\in E\backslash\{0\}|\langle I'(u),u\rangle=0\}$ is the Nehari manifold.
\end{theorem}

Thirdly,
we prove the existence of special solution for equation \eqref{P}, which satisfies Poho\v{z}aev identity.
\begin{theorem}\label{Theorem1.3}
Let $N\geqslant3$, $s\in(0,1)$ and $p\in (2^{*}_{s},2^{*})$.
Then equation \eqref{P} has a solution $\Psi\in E$, which satisfies
\begin{equation*}
\begin{aligned}
P(\Psi)=\frac{N-2}{2}
\|\Psi\|^{2}_{D^{1,2}(\mathbb{R}^{N})}
+
\frac{N-2s}{2}
\|\Psi\|^{2}_{D^{s,2}(\mathbb{R}^{N})}
-\frac{N}{p}\|\Phi\|_{L^{p}(\mathbb{R}^{N})}^{p}=0,
\end{aligned}
\end{equation*}
and
\begin{equation*}
\begin{aligned}
I(\Psi)=\inf_{u\in \mathcal{M}}I(u)=c
\end{aligned}
\end{equation*}
where $\mathcal{M}:=\{u\in E\backslash\{0\}|P(u)=0\}$ is the Nehari-Poho\v{z}aev manifold.
\end{theorem}

\begin{remark}
Due to the appearance of the mixed local and nonlocal operator and the loss of $L^{2}$ term in equation \eqref{P}, it is hard to get the regularity of weak solutions.
Hence, it is not easy to show that all weak solutions of equation \eqref{P} satisfy Poho\v{z}aev identy.
Nevertheless,
combining Theorems \ref{Theorem1.2} and \ref{Theorem1.3}, we know that the ground sate solutions of equation \eqref{P} satisfy Poho\v{z}aev identy.
%This partial answer the question of (iii) in Remark 1.1.
\end{remark}

Finally, by using Theorem \ref{Theorem1.3}, we establish the best constant $C_{N,p,s}$.
\begin{theorem}\label{Theorem1.4}
Let $N\geqslant3$, $s\in(0,1)$ and $p\in (2^{*}_{s},2^{*})$. Then we have the following results.

\begin{enumerate}
\item [(1)]
The best constant of \eqref{1.1} is
\begin{equation*}
\begin{aligned}
C_{N,p,s}^{-1}
=\left(
\frac{2N-p(N-2)}{2p(1-s)}
\right)
^{\frac{2N-p(N-2)}{4(1-s)}}
\left(
\frac{p(N-2s)-2N}{2p(1-s)}
\right)
^{\frac{p(N-2s)-2N}{4(1-s)}}
\left[
\frac{2p}{p-2}
c
\right]^{\frac{p-2}{2}}
\end{aligned}
\end{equation*}
where $c=\inf\limits_{u\in \mathcal{N}}I(u)$.

\item [(2)]
The optimizer $Q$ is not only a weak solution of equation \eqref{P}, but it is also a ground state solution and satisfies
\begin{equation*}
\begin{aligned}
I(Q)=c~~\mathrm{and}~~P(Q)=0.
\end{aligned}
\end{equation*}
\end{enumerate}
\end{theorem}

This paper is organized as follows. In Section 2, we collect basic properties of the function
space $E$ and establish an extanded Lieb translation result.
In Section 3-6, we prove Theorems \ref{Theorem1.1}-\ref{Theorem1.4},
respectively.

\section{Sobolev Space and Lieb's translation theorem}
The homogeneous Sobolev space
\begin{equation*}
\begin{aligned}
D^{1,2}(\mathbb{R}^{N})
=
\{u\in L^{2^{*}}(\mathbb{R}^{N})|
|\nabla u|\in L^{2}(\mathbb{R}^{N})\},
\end{aligned}
\end{equation*}
its semi-norm is taken as
\begin{equation*}
\begin{aligned}
\|u\|_{D^{1,2}
(\mathbb{R}^{N})}^{2}
=\int_{\mathbb{R}^{N}}
|\nabla u|^{2}
\mathrm{d}x.
\end{aligned}
\end{equation*}
The homogeneous fractional Sobolev space $D^{s,2}(\mathbb{R}^{N})$ is the completion of $C_{0}^{\infty}(\mathbb{R}^{N})$ under the semi-norm
\begin{equation*}
\begin{aligned}
\|u\|_{D^{s,2}(\mathbb{R}^{N})}^{2}:=
\int_{\mathbb{R}^{N}}\int_{\mathbb{R}^{N}}\frac{|u(x)-u(y)|^{2}}{|x-y|^{N+2s}}\mathrm{d}x\mathrm{d}y.
\end{aligned}
\end{equation*}
The mixed Sobolev space $E$ is defined by the completion of $C^{\infty}_{0}(\mathbb{R}^{N})$ under the semi-norm
\begin{equation*}
\begin{aligned}
\|u\|_{E}^{2}:=
\int_{\mathbb{R}^{N}}
|\nabla u|^{2}
\mathrm{d}x
+
\int_{\mathbb{R}^{N}}
\int_{\mathbb{R}^{N}}
\frac{|u(x)-u(y)|^{2}}{|x-y|^{N+2s}}
\mathrm{d}x
\mathrm{d}y.
\end{aligned}
\end{equation*}

\begin{lemma}\label{Lemma2.1}
$E\hookrightarrow D^{1,2}(\mathbb{R}^{N})$ and
$E\hookrightarrow D^{s,2}(\mathbb{R}^{N})$.
\end{lemma}
\begin{proof}
It is easy to see that
\begin{equation*}
\begin{aligned}
\|u\|_{D^{1,2}(\mathbb{R}^{N})}^{2}
\leqslant \|u\|_{E}^{2},
\end{aligned}
\end{equation*}
and
\begin{equation*}
\begin{aligned}
\|u\|_{D^{s,2}(\mathbb{R}^{N})}^{2}
\leqslant \|u\|_{E}^{2}.
\end{aligned}
\end{equation*}
Therefore $E\hookrightarrow D^{1,2}(\mathbb{R}^{N})$ and
$E\hookrightarrow D^{s,2}(\mathbb{R}^{N})$.
\end{proof}

\begin{lemma}\label{Lemma2.2}
$E\hookrightarrow L^{t}(\mathbb{R}^{N})$, $t\in [\frac{2N}{N-2s},\frac{2N}{N-2}]$.
\end{lemma}

\begin{proof}
Using H\"{o}lder's inequality, we have
\begin{equation}\label{2.1}
\begin{aligned}
\int_{\mathbb{R}^{N}}
|u|^{t}
\mathrm{d}x
\leqslant&
\left(
\int_{\mathbb{R}^{N}}
|u|^{\frac{2N}{N-2s}}
\mathrm{d}x
\right)
^{\frac{(N-2s)[2N-t(N-2)]}{4N(1-s)}}
\left(
\int_{\mathbb{R}^{N}}
|u|^{\frac{2N}{N-2}}
\mathrm{d}x
\right)
^{\frac{(N-2)[t(N-2s)-2N]}{4N(1-s)}}.
\end{aligned}
\end{equation}
From Lemma \ref{Lemma2.1}, we know that:
\begin{equation*}
\begin{aligned}
\left(
\int_{\mathbb{R}^{N}}
|u|^{\frac{2N}{N-2s}}
\mathrm{d}x
\right)^{\frac{2}{2_{s}^{*}}}
\leqslant
\|u\|_{D^{s,2}(\mathbb{R}^{N})}^{2}
\leqslant \|u\|_{E}^{2},
\end{aligned}
\end{equation*}
and
\begin{equation*}
\begin{aligned}
\left(
\int_{\mathbb{R}^{N}}
|u|^{\frac{2N}{N-2}}
\mathrm{d}x
\right)^{\frac{2}{2^{*}}}
\leqslant
\|u\|_{D^{1,2}(\mathbb{R}^{N})}^{2}
\leqslant \|u\|_{E}^{2}.
\end{aligned}
\end{equation*}
Then we get
\begin{equation*}
\begin{aligned}
\int_{\mathbb{R}^{N}}
|u|^{t}
\mathrm{d}x
\leqslant&
\left(
\int_{\mathbb{R}^{N}}
|u|^{\frac{2N}{N-2s}}
\mathrm{d}x
\right)
^{\frac{(tN-2N-2t)(N-2s)}{4N(s-1)}}
\left(
\int_{\mathbb{R}^{N}}
|u|^{\frac{2N}{N-2}}
\mathrm{d}x
\right)
^{\frac{(tN-2N-2ts)(N-2)}{4N(1-s)}}\\
\leqslant&\|u\|_{E}^{t}<\infty.
\end{aligned}
\end{equation*}
The proof is complete.
\end{proof}

\begin{lemma}\label{Lemma2.3}
\cite[Theorem 1.1]{Cotsiolis-Tavoularis2004JMAA}
Let
$s\in(0,1]$
and $N>2s$.
Then there exists a constant
$S_{s}>0$
such that for
any
$u\in D^{s,2}(\mathbb{R}^{N})$,
\begin{equation*}
\begin{aligned}
\|u\|_{L^{2_{s}^{*}}(\mathbb{R}^{N})}^{2}
\leqslant
S_{s}^{-1}
\|u\|^{2}_{D^{s,2}}(\mathbb{R}^{N}).
\end{aligned}
\end{equation*}
\end{lemma}

\begin{lemma}\label{Lemma2.4}
Let $N\geqslant 3$, $s\in(0,1)$ and $q\in(2^{*}_{s},2^{*})$. Then the following inequality holds
\begin{equation*}
\begin{aligned}
\int_{\mathbb{R}^{N}}|u|^{q}\mathrm{d}x\leqslant2C(N+1)^{2}
\left(\sup_{z\in \mathbb{R}^{N}}
\int_{B(z,1)}
|u|^{q}
\mathrm{d}x\right)^{\frac{q-2}{q}}
\|u\|_{E}^{2},
\end{aligned}
\end{equation*}
for all $u\in E$.
\end{lemma}

\begin{proof}
Let $u\in E$ and $q\in(2_{s}^{*},2^{*})$.
From H\"{o}lder's inequality and Lemma \ref{Lemma2.3}, we have
\begin{equation}\label{2.2}
\begin{aligned}
&\int_{B(z,1)}|u|^{q}\mathrm{d}x\\
\leqslant&
C
\left(
\int_{B(z,1)}\int_{B(z,1)}\frac{|u(x)-u(y)|^{2}}{|x-y|^{N+2s}}\mathrm{d}x\mathrm{d}y
+
\int_{B(z,1)}
|\nabla u|^{2}
\mathrm{d}x\right)^{\frac{q}{2}}.
\end{aligned}
\end{equation}
Applying \eqref{2.2}, we know that
\begin{equation*}
\begin{aligned}
&\int_{B(z,1)}|u|^{q}\mathrm{d}x\\
=&
\left(\int_{B(z,1)}|u|^{q}\mathrm{d}x\right)^{\frac{2}{q}}
\left(\int_{B(z,1)}|u|^{q}\mathrm{d}x\right)
^{\frac{q-2}{q}}\\
\leqslant&
C
\left(
\int_{B(z,1)}\int_{B(z,1)}\frac{|u(x)-u(y)|^{2}}{|x-y|^{N+2s}}\mathrm{d}x\mathrm{d}y
+
\int_{B(z,1)}
|\nabla u|^{2}
\mathrm{d}x\right)
\left(\int_{B(z,1)}|u|^{q}\mathrm{d}x\right)
^{\frac{q-2}{q}}.
\end{aligned}
\end{equation*}
Covering $\mathbb{R}^{N}$ by balls of radius $1$,
in such a way that each point of $\mathbb{R}^{N}$ is contained in at most $N+1$ balls.
We get
\begin{equation*}
\begin{aligned}
&\int_{\mathbb{R}^{N}}|u|^{q}\mathrm{d}x\\
\leqslant& C(N+1)
\left(\sup_{z\in \mathbb{R}^{N}}
\int_{B(z,1)}
|u|^{q}
\mathrm{d}x\right)^{\frac{q-2}{q}}
\left(
\int_{\mathbb{R}^{N}}\int_{\mathbb{R}^{N}}\frac{|u(x)-u(y)|^{2}}{|x-y|^{N+2s}}\mathrm{d}y\mathrm{d}x
+
\int_{\mathbb{R}^{N}}
|\nabla u|^{2}
\mathrm{d}x\right).
\end{aligned}
\end{equation*}
\end{proof}

As an application of Lemma \ref{Lemma2.4}, we present an extansion of Lieb's translation theorem.

\begin{lemma}\label{Lemma2.5}
Let $N\geqslant3$, $0<s<1$, $q\in (2^{*}_{s},2^{*})$, and $\{u_{n}\}$ be a bounded sequence in $E$ satisfing  $\lim\limits_{n\to\infty}\int_{\mathbb{R}^{N}}|u_{n}|^{q}\mathrm{d}x>0$.
Then there exists $\{z_{n}\}\subset \mathbb{R}^{N}$ such that $\{\bar{u}_{n}:=u_{n}(x+z_{n})\}$ convergence strongly to $\bar{u}\not\equiv0$ in $L^{q}_{loc}(\mathbb{R}^{N})$.
\end{lemma}

\begin{proof}
Note that $\{u_{n}\}$ is a bounded sequence in $E$.
Up to a subsequence,
we may assume that
\begin{equation*}
\begin{aligned}
u_{n}\rightharpoonup u
~
\mathrm{in}
~
E,~
u_{n}\rightarrow u
~
\mathrm{a.e. ~in}
~\mathbb{R}^{N},
u_{n}\rightarrow u
~
\mathrm{in}
~
L^{q}_{loc}(\mathbb{R}^{N}).
\end{aligned}
\end{equation*}
Applying Lemma \ref{Lemma2.4} and $\lim\limits_{n\to\infty}\int_{\mathbb{R}^{N}}|u_{n}|^{q}\mathrm{d}x>0$,
there exists
$C>0$
such that
\begin{equation*}
\begin{aligned}
\sup_{z\in \mathbb{R}^{N}}
\int_{B(z,1)}
|u|^{q}
\mathrm{d}x\geqslant C>0.
\end{aligned}
\end{equation*}
Note that $\{u_{n}\}$ is bounded in $E$ and $E\hookrightarrow L^{q}(\mathbb{R}^{N})$,
we have
\begin{equation*}
\begin{aligned}
\sup_{z\in \mathbb{R}^{N}}
\int_{B(z,1)}
|u|^{q}
\mathrm{d}x\leqslant\int_{\mathbb{R}^{N}}
|u_{n}|^{q}
\mathrm{d}x \leqslant C.
\end{aligned}
\end{equation*}
Hence, there exists $C_{0}$ such that
\begin{equation*}
\begin{aligned}
C_{0}
\leqslant
\sup_{z\in \mathbb{R}^{N}}
\int_{B(z,1)}
|u|^{q}
\mathrm{d}x\leqslant C_{0}^{-1}.
\end{aligned}
\end{equation*}
From the above inequality,
there exists
$z_{n}\in \mathbb{R}^{N}$
such that
\begin{equation*}
\begin{aligned}
\int_{B(z_{n},1)}
|u_{n}|^{q}
\mathrm{d}x
\geqslant
\sup_{z\in \mathbb{R}^{N}}
\int_{B(z,1)}
|u_{n}|^{q}
\mathrm{d}x
-
\frac{C}{2n}
\geqslant
C_{1}>0.
\end{aligned}
\end{equation*}
Set $\bar{u}_{n}:=u_{n}(x+z_{n})$. Then $\|\bar{u}_{n}\|_{E}=\|u_{n}\|_{E}$ and
\begin{equation*}
\begin{aligned}
\int_{B(0,1)}
|\bar{u}_{n}|^{q}
\mathrm{d}x
\geqslant
C_{1}>0.
\end{aligned}
\end{equation*}
Up to a subsequence,
there exists
$\bar{u}$
such that
\begin{equation*}
\begin{aligned}
\bar{u}_{n}\rightharpoonup \bar{u}
\;
\mathrm{in}
~
E,
\;\;
\bar{u}_{n}\rightarrow \bar{u}
~
\mathrm{a.e. ~in}
~
\mathbb{R}^{N}.
\end{aligned}
\end{equation*}
Using the compact embedding
$E
\hookrightarrow
L^{q}_{loc}(\mathbb{R}^{N})$,
we deduce that
$\bar{u}\not\equiv0$.
\end{proof}

\section{The proof of Theorem \ref{Theorem1.1}}
In this section, we show the Gagliardo-Nirenberg inequality
and the existence of optimizer for it.
\begin{lemma}\label{Lemma3.1}
Let $N\geqslant3$, $s\in(0,1)$ and $p\in (2^{*}_{s},2^{*})$.
For all $u\in E$, there exists $C_{N,p,s}>0$ (independent of $u$) such that
\begin{equation*}
\begin{aligned}
\int_{\mathbb{R}^{N}}
|u|^{p}
\mathrm{d}x
\leqslant
C_{N,p,s}
\left(
\int_{\mathbb{R}^{N}}
\int_{\mathbb{R}^{N}}
\frac{|u(x)-u(y)|^{2}}{|x-y|^{N+2s}}
\mathrm{d}x
\mathrm{d}y
\right)
^{\frac{2N-p(N-2)}{4(1-s)}}
\left(
\int_{\mathbb{R}^{N}}
|\nabla u|^{2}
\mathrm{d}x
\right)
^{\frac{p(N-2s)-2N}{4(1-s)}}.
\end{aligned}
\end{equation*}
\end{lemma}
\begin{proof}
From Lemma \ref{Lemma2.3}, we have
\begin{equation}\label{3.1}
\begin{aligned}
\left(
\int_{\mathbb{R}^{N}}
|u|^{\frac{2N}{N-2s}}
\mathrm{d}x
\right)
^{\frac{N-2s}{N}}
\leqslant&
S_{s}^{-1}
\int_{\mathbb{R}^{N}}
\int_{\mathbb{R}^{N}}
\frac{|u(x)-u(y)|^{2}}{|x-y|^{N+2s}}
\mathrm{d}x
\mathrm{d}y,
\end{aligned}
\end{equation}
and
\begin{equation}\label{3.2}
\begin{aligned}
\left(
\int_{\mathbb{R}^{N}}
|u|^{\frac{2N}{N-2}}
\mathrm{d}x
\right)
^{\frac{N-2}{N}}
\leqslant&
S_{1}^{-1}
\int_{\mathbb{R}^{N}}
|\nabla u|^{2}
\mathrm{d}x.
\end{aligned}
\end{equation}
Plugging \eqref{3.1} and \eqref{3.2} into \eqref{2.1}, we deduce that
\begin{equation*}
\begin{aligned}
\int_{\mathbb{R}^{N}}
|u|^{p}
\mathrm{d}x
\leqslant&
\left(
\int_{\mathbb{R}^{N}}
|u|^{\frac{2N}{N-2s}}
\mathrm{d}x
\right)
^{\frac{N-2s}{N}\cdot\frac{2N-p(N-2)}{4(1-s)}}
\left(
\int_{\mathbb{R}^{N}}
|u|^{\frac{2N}{N-2}}
\mathrm{d}x
\right)
^{\frac{N-2}{N}\cdot
\frac{p(N-2s)-2N}{4(1-s)}}\\
\leqslant&
C
\left(
\int_{\mathbb{R}^{N}}
\int_{\mathbb{R}^{N}}
\frac{|u(x)-u(y)|^{2}}{|x-y|^{N+2s}}
\mathrm{d}x
\mathrm{d}y
\right)
^{\frac{2N-p(N-2)}{4(1-s)}}
\left(
\int_{\mathbb{R}^{N}}
|\nabla u|^{2}
\mathrm{d}x
\right)
^{\frac{p(N-2s)-2N}{4(1-s)}}.
\end{aligned}
\end{equation*}
\end{proof}

We define the following Weinstein functional
\begin{equation*}
\begin{aligned}
W(u):=
\frac{\left(
\int_{\mathbb{R}^{N}}
\int_{\mathbb{R}^{N}}
\frac{|u(x)-u(y)|^{2}}{|x-y|^{N+2s}}
\mathrm{d}x
\mathrm{d}y
\right)
^{\frac{2N-p(N-2)}{4(1-s)}}
\left(
\int_{\mathbb{R}^{N}}
|\nabla u|^{2}
\mathrm{d}x
\right)
^{\frac{p(N-2s)-2N}{4(1-s)}}}
{\int_{\mathbb{R}^{N}}
|u|^{p}
\mathrm{d}x}.
\end{aligned}
\end{equation*}
Then
\begin{equation*}
\begin{aligned}
\inf_{u\in E\backslash\{0\}}W(u)=C_{N,p,s}^{-1}.
\end{aligned}
\end{equation*}

\begin{lemma}\label{Lemma3.2}
Let $N\geqslant3$, $s\in(0,1)$ and $p\in (2^{*}_{s},2^{*})$.
There exists a radial minimzing sequence $\{u_{n}\}\subset E$ of $C_{N,p,s}^{-1}$ such that
\begin{equation*}
\begin{aligned}
\|u_{n}\|_{D^{s,2}(\mathbb{R}^{N})}^{2}
=
\|u_{n}\|_{D^{1,2}(\mathbb{R}^{N})}^{2}
=1.
\end{aligned}
\end{equation*}
\end{lemma}

\begin{proof}
{\bf Step 1.}
From Lemma \ref{Lemma3.1}, we know that there exists a minimzing sequence $\{w_{n}\}\subset E$ such that
\begin{equation*}
\begin{aligned}
\lim\limits_{n\to\infty}
W(w_{n})=C_{N,p,s}^{-1}.
\end{aligned}
\end{equation*}
It is easy to see that $\{|w_{n}|\}\subset E$ and $\lim\limits_{n\to\infty}
W(|w_{n}|)=C_{N,p,s}^{-1}$.

Let $|w_{n}|^{*}$ be the symmetric decreasing rearrangement of $|w_{n}|$.
From [14] and [21], one has
\begin{equation*}
\begin{aligned}
\||w_{n}|^{*}\|_{D^{s,2}(\mathbb{R}^{N})}^{2}
\leqslant
\||w_{n}|\|_{D^{1,2}(\mathbb{R}^{N})}^{2},
\end{aligned}
\end{equation*}
and
\begin{equation*}
\begin{aligned}
\||w_{n}|^{*}\|_{D^{s,2}(\mathbb{R}^{N})}^{2}
\leqslant
\||w_{n}|\|_{D^{1,2}(\mathbb{R}^{N})}^{2},
\end{aligned}
\end{equation*}
and
\begin{equation*}
\begin{aligned}
\int_{\mathbb{R}^{N}}
||w_{n}|^{*}|^{p}
\mathrm{d}x
=
\int_{\mathbb{R}^{N}}
|w_{n}|^{p}
\mathrm{d}x.
\end{aligned}
\end{equation*}
Hence
$C_{N,p,s}^{-1}\leqslant\lim\limits_{n\to\infty}
W(|w_{n}|^{*})\leqslant\lim\limits_{n\to\infty}
W(|w_{n}|)=C_{N,p,s}^{-1}$.
Set
\begin{equation*}
\begin{aligned}
v_{n}=|w_{n}|^{*}.
\end{aligned}
\end{equation*}
Then $\lim\limits_{n\to\infty}
W(v_{n})=C_{N,p,s}^{-1}$.

{\bf Step 2.}
We rescale the sequence $\{v_{n}\}$ by
\begin{equation*}
\begin{aligned}
u_{n}=\lambda_{1,n}v_{n}(\lambda_{2,n}x),
\end{aligned}
\end{equation*}
where
\begin{equation*}
\begin{aligned}
\lambda_{1,n}
=
\frac
{\left[\int_{\mathbb{R}^{N}}
\int_{\mathbb{R}^{N}}
\frac{|v_{n}(x)-v_{n}(y)|^{2}}{|x-y|^{N+2s}}
\mathrm{d}x
\mathrm{d}y
\right]^{\frac{N-2}{1-s}}}
{\left[\int_{\mathbb{R}^{N}}
|\nabla v_{n}(x)|^{2}
\mathrm{d}x
\right]^{\frac{N-2s}{1-s}}},
\end{aligned}
\end{equation*}
and
\begin{equation*}
\begin{aligned}
\lambda_{2,n}
=
\left[
\frac
{\int_{\mathbb{R}^{N}}
\int_{\mathbb{R}^{N}}
\frac{|v_{n}(x)-v_{n}(y)|^{2}}{|x-y|^{N+2s}}
\mathrm{d}x
\mathrm{d}y}
{\int_{\mathbb{R}^{N}}
|\nabla v_{n}(x)|^{2}
\mathrm{d}x}
\right]^{\frac{1}{2-2s}}.
\end{aligned}
\end{equation*}
Then
\begin{equation*}
\begin{aligned}
1=&
\int_{\mathbb{R}^{N}}
\int_{\mathbb{R}^{N}}
\frac{|u_{n}(x)-u_{n}(y)|^{2}}{|x-y|^{N+2s}}
\mathrm{d}x
\mathrm{d}y\\
=&
\lambda_{1,n}^{2}
\lambda_{2,n}^{2s-N}
\int_{\mathbb{R}^{N}}
\int_{\mathbb{R}^{N}}
\frac{|v_{n}(x)-v_{n}(y)|^{2}}{|x-y|^{N+2s}}
\mathrm{d}x
\mathrm{d}y,
\end{aligned}
\end{equation*}
and
\begin{equation*}
\begin{aligned}
1=
\int_{\mathbb{R}^{N}}
|\nabla u_{n}|^{2}
\mathrm{d}x
=&
\lambda_{1,n}^{2}
\lambda_{2,n}^{2-N}
\int_{\mathbb{R}^{N}}
|\nabla v_{n}(x)|^{2}
\mathrm{d}x,
\end{aligned}
\end{equation*}
and
\begin{equation*}
\begin{aligned}
\int_{\mathbb{R}^{N}}
|u_{n}|^{p}
\mathrm{d}x
=&\lambda_{1,n}^{p}
\lambda_{2,n}^{-N}
\int_{\mathbb{R}^{N}}
|v_{n}(x)|^{p}
\mathrm{d}x.
\end{aligned}
\end{equation*}
{\bf Step 3.}
The Weinstein functional $W(\cdot)$ is invariant with respect to our scaling:
\begin{equation*}
\begin{aligned}
W(u_{n})
=&
\frac{\left(
\|u_{n}\|_{D^{s,2}(\mathbb{R}^{N})}^{2}
\right)
^{\frac{2N-p(N-2)}{4(1-s)}}
\left(
\|u_{n}\|_{D^{1,2}(\mathbb{R}^{N})}^{2}
\right)
^{\frac{p(N-2s)-2N}{4(1-s)}}}
{\int_{\mathbb{R}^{N}}
|u|^{p}
\mathrm{d}x}\\
=&
\frac{[\lambda_{1,n}^{2}
\lambda_{2,n}^{2s-N}]^{\frac{2N-p(N-2)}{4(1-s)}}[\lambda_{1,n}^{2}
\lambda_{2,n}^{2-N}]^{\frac{p(N-2s)-2N}{4(1-s)}}}{\lambda_{1,n}^{p}
\lambda_{2,n}^{-N}}
W(v_{n})\\
=&W(v_{n}).
\end{aligned}
\end{equation*}
This impiles that $\{u_{n}\}\subset E$ is a radial bounded minimizing sequence.
\end{proof}

\begin{lemma}\label{Lemma3.3}
Let $N\geqslant3$, $s\in(0,1)$ and $p\in (2^{*}_{s},2^{*})$.
Let $\{u_{n}\}\subset E$ be a radial minimzing sequence  of $C_{N,p,s}^{-1}$ which is in Lemma \ref{Lemma3.2}.
Then $u_{n}\to u\not\equiv 0$ in $E$, and
\begin{equation*}
\begin{aligned}
\|u\|_{D^{s,2}(\mathbb{R}^{N})}^{2}
=
\|u\|_{D^{1,2}(\mathbb{R}^{N})}^{2}
=1
\end{aligned}
\end{equation*}
and
\begin{equation*}
\begin{aligned}
W(u)=C_{N,p,s}^{-1}.
\end{aligned}
\end{equation*}
\end{lemma}

\begin{proof}
{\bf Step 1.}
From Lemma \ref{Lemma3.2}, we have $\|u_{n}\|_{D^{s,2}(\mathbb{R}^{N})}^{2}
=
\|u_{n}\|_{D^{1,2}(\mathbb{R}^{N})}^{2}
=1$ and
\begin{equation*}
\begin{aligned}
0<C_{N,p,s}
=&\lim\limits_{n\to\infty}
\frac
{\int_{\mathbb{R}^{N}}
|u_{n}|^{p}
\mathrm{d}x}
{\left(
\int_{\mathbb{R}^{N}}
\int_{\mathbb{R}^{N}}
\frac{|u_{n}(x)-u_{n}(y)|^{2}}{|x-y|^{N+2s}}
\mathrm{d}x
\mathrm{d}y
\right)
^{\frac{2N-p(N-2)}{4(1-s)}}
\left(
\int_{\mathbb{R}^{N}}
|\nabla u_{n}|^{2}
\mathrm{d}x
\right)
^{\frac{p(N-2s)-2N}{4(1-s)}}}\\
=&\lim\limits_{n\to\infty}
\int_{\mathbb{R}^{N}}
|u_{n}|^{p}
\mathrm{d}x.
\end{aligned}
\end{equation*}
Combining $\lim\limits_{n\to\infty}
\int_{\mathbb{R}^{N}}
|u_{n}|^{p}
\mathrm{d}x\geqslant C>0$ and Lemma \ref{Lemma2.4}, one has that there exists a sequence of $\{z_{n}\}\subset \mathbb{R}^{N}$ such that
\begin{equation}\label{3.3}
\begin{aligned}
\int_{B(z_{n},1)}
|u_{n}|^{p}
\mathrm{d}x
\geqslant C_{1}>0.
\end{aligned}
\end{equation}
where $C_{1}>0$ is a constant independent of $u_{n}$.

{\bf Step 2.}
We show that $\{z_{n}\}$ is a bounded sequence in $\mathbb{R}^{N}$.
Suppose on the contrary that
$|z_{n}|\rightarrow\infty$
as
$n\rightarrow\infty$.
We have
\begin{align*}
\sup_{|x|>0}
|u_{n}(x)|
\leqslant
\frac{C}{|x|^{\frac{N-2}{2}}}
\|u_{n}\|_{D^{1,2}(\mathbb{R}^{N})}
\leqslant
\frac{C}{|x|^{\frac{N-2}{2}}}.
\end{align*}
For any
$\left[\frac{C_{1}}{|B(0,1)|}\right]^{\frac{1}{p}}>\varepsilon>0$,
there exists an $M>0$ such that for any $n>M$ we can deduce that
\begin{align*}
|u_{n}(x)|
\leqslant
\frac{C}
{||x_{n}|-1|^{\frac{N-2}{2}}}
\leqslant
\varepsilon,
\ ~x\in B^{c}(0,||x_{n}|-1|).
\end{align*}
Note that
$B(z_{n},1)\subset B^{c}(0,||x_{n}|-1|)$.
Then
\begin{align*}
\int_{B(z_{n},1)}
|u_{n}|^{p}
\mathrm{d}x
\leqslant
\varepsilon^{p}
\int_{B(z_{n},1)}
\mathrm{d}x
=
\varepsilon^{p}
|B(z_{n},1)|
=
\varepsilon^{p}
|B(0,1)|
<C_{1}.
\end{align*}
This contradicts \eqref{3.3}.
Then we know that there exists $C>0$ such that $B(z_{n},1)\subset B(0,C)$ and
\begin{equation*}
\begin{aligned}
0<C_{1}\leqslant\int_{B(z_{n},1)}
|u_{n}|^{p}
\mathrm{d}x
\leqslant
\int_{B(0,C)}
|u_{n}|^{p}
\mathrm{d}x.
\end{aligned}
\end{equation*}
From $E\hookrightarrow\hookrightarrow L^{p}(B(0,C))$, we know that
\begin{equation*}
\begin{aligned}
u_{n}\rightharpoonup u\not\equiv 0.
\end{aligned}
\end{equation*}

{\bf Step 3.}
It follows from Brézis-Lieb Lemma, Lemma \ref{Lemma3.1} and Young's inequality that
\begin{equation*}
\begin{aligned}
0<&C_{N,p,s}\\
=&\lim\limits_{n\to\infty}
\int_{\mathbb{R}^{N}}
|u_{n}|^{p}
\mathrm{d}x\\
=&
\int_{\mathbb{R}^{N}}
|u|^{p}
\mathrm{d}x
+\lim\limits_{n\to\infty}
\int_{\mathbb{R}^{N}}
|u_{n}-u|^{p}
\mathrm{d}x\\
=&
C_{N,p,s}
\|u\|_{D^{s,2}(\mathbb{R}^{N})}^{\frac{2N-p(N-2)}{2(1-s)}}
\|u\|_{D^{1,2}(\mathbb{R}^{N})}^{\frac{p(N-2s)-2N}{2(1-s)}}\\
&+\lim\limits_{n\to\infty}
C_{N,p,s}
\|u_{n}-u\|_{D^{s,2}(\mathbb{R}^{N})}^{\frac{2N-p(N-2)}{2(1-s)}}
\|u_{n}-u\|_{D^{1,2}(\mathbb{R}^{N})}^{\frac{p(N-2s)-2N}{2(1-s)}}\\
\leqslant&
C_{N,p,s}
\left[
\frac{2N-p(N-2)}{2p(1-s)}
\|u\|_{D^{s,2}(\mathbb{R}^{N})}^{p}
+\frac{p(N-2s)-2N}{2p(1-s)}
\|u\|_{D^{1,2}(\mathbb{R}^{N})}^{p}
\right]\\
&+\lim\limits_{n\to\infty}
C_{N,p,s}
\left[
\frac{2N-p(N-2)}{2p(1-s)}
\|u_{n}-u\|_{D^{s,2}(\mathbb{R}^{N})}^{p}
+\frac{p(N-2s)-2N}{2p(1-s)}
\|u_{n}-u\|_{D^{1,2}(\mathbb{R}^{N})}^{p}
\right]\\
\leqslant&
C_{N,p,s}
\frac{2N-p(N-2)}{2p(1-s)}
\left[
\|u\|_{D^{s,2}(\mathbb{R}^{N})}^{2}
+
\|u_{n}-u\|_{D^{s,2}(\mathbb{R}^{N})}^{2}
\right]^{\frac{p}{2}}\\
&+C_{N,p,s}
\frac{p(N-2s)-2N}{2p(1-s)}
\lim\limits_{n\to\infty}
\left[
\|u\|_{D^{1,2}(\mathbb{R}^{N})}^{2}
+
\|u_{n}-u\|_{D^{1,2}(\mathbb{R}^{N})}^{2}
\right]^{\frac{p}{2}}\\
=&C_{N,p,s}
\frac{2N-p(N-2)}{2p(1-s)}
+C_{N,p,s}
\frac{p(N-2s)-2N}{2p(1-s)}\\
=&C_{N,p,s}.
\end{aligned}
\end{equation*}
Therefore all the above inequalities must be equalities. Hence,
\begin{equation*}
\begin{aligned}
\|u\|_{D^{1,2}(\mathbb{R}^{N})}^{2}=0~
\mathrm{or}~~\|u_{n}-u\|_{D^{1,2}(\mathbb{R}^{N})}^{2}=0.
\end{aligned}
\end{equation*}
Note that $u\not\equiv0$. We know that $\|u\|_{D^{1,2}(\mathbb{R}^{N})}^{2}\not=0$ and $\|u_{n}-u\|_{D^{1,2}(\mathbb{R}^{N})}^{2}=0$ and
\begin{equation*}
\begin{aligned}
\|u\|_{D^{s,2}(\mathbb{R}^{N})}^{2}
=
\|u\|_{D^{1,2}(\mathbb{R}^{N})}^{2}
=1
\end{aligned}
\end{equation*}
and
\begin{equation*}
\begin{aligned}
W(u)=C_{N,p,s}^{-1}
\end{aligned}
\end{equation*}
Thus $u$ is a minimizer of $C_{N,p,s}^{-1}$.
\end{proof}

\begin{lemma}\label{Lemma3.4}
Let $N\geqslant3$, $s\in(0,1)$ and $p\in (2^{*}_{s},2^{*})$.
Let $u$ be a minimizer of $C_{N,p,s}^{-1}$.
Then $Q=\lambda_{1}u(\lambda_{2}x)$ is a minimizer of $C_{N,p,s}^{-1}$, and also is a weak solution of
\begin{equation*}
\begin{aligned}
-\Delta Q
+(-\Delta)^{s}Q =
|Q|^{p-2}Q, \ \ x\in\mathbb{R}^{N},
\end{aligned}
\end{equation*}
where
\begin{equation*}
\begin{aligned}
\begin{cases}
\lambda_{1}=
\left[
\frac{p(N-2s)-2N}{2N-p(N-2)}
\right]^{\frac{1}{(1-s)(p-2)}}
\left[
\frac{2p(1-s)}{p(N-2s)-2N}
\frac{1}{\|u\|_{L^{p}(\mathbb{R}^{N})}^{p}}
\right]^{\frac{1}{p-2}},\\
\lambda_{2}=
\left[
\frac{p(N-2s)-2N}{2N-p(N-2)}
\right]^{\frac{1}{2-2s}}.
\end{cases}
\end{aligned}
\end{equation*}

\end{lemma}

\begin{proof}
From Lemma \ref{Lemma3.3}, one has that
$u$ is minimizer of $W(u)=C_{N,p,s}^{-1}$.
Then
\begin{equation*}
\begin{aligned}
\left.\frac{\mathrm{d}}
{\mathrm{d}\varepsilon}\right|_{\varepsilon=0}W(u+\varepsilon \varphi)=0,
\end{aligned}
\end{equation*}
for all $\varphi\in E$.
Thus, $u$ is a weak solution of the Euler-Lagrange equation
\begin{equation*}
\begin{aligned}
&-\Delta u
+\frac{2N-p(N-2)}{p(N-2s)-2N}
\frac{
\|u\|_{D^{1,2}(\mathbb{R}^{N})}^{2}
}
{\|u\|_{D^{s,2}(\mathbb{R}^{N})}^{2}}
(-\Delta)^{s}u\\
=&
\frac{2p(1-s)}{p(N-2s)-2N}
\frac{\|u\|_{D^{1,2}(\mathbb{R}^{N})}^{2}}
{\|u\|_{L^{p}(\mathbb{R}^{N})}^{p}}
|u|^{p-2}u, \ \ x\in\mathbb{R}^{N},
\end{aligned}
\end{equation*}
Set $Q=\lambda_{1}u(\lambda_{2}x)$, where
\begin{equation*}
\begin{aligned}
\begin{cases}
\lambda_{2}=
\left[
\frac{p(N-2s)-2N}{2N-p(N-2)}
\frac{
\|u\|_{D^{s,2}(\mathbb{R}^{N})}^{2}
}
{\|u\|_{D^{1,2}(\mathbb{R}^{N})}^{2}}\right]^{\frac{1}{2-2s}}\\
\lambda_{1}=\left[
\frac{p(N-2s)-2N}{2N-p(N-2)}
\frac{
\|u\|_{D^{s,2}(\mathbb{R}^{N})}^{2}
}
{\|u\|_{D^{1,2}(\mathbb{R}^{N})}^{2}}\right]^{\frac{1}{(1-s)(p-2)}}
\left[
\frac{2p(1-s)}{p(N-2s)-2N}
\frac{\|u\|_{D^{1,2}(\mathbb{R}^{N})}^{2}}
{\|u\|_{L^{p}(\mathbb{R}^{N})}^{p}}
\right]^{\frac{1}{p-2}}
\end{cases}
\end{aligned}
\end{equation*}
Then $Q$ is a weak solution of
\begin{equation*}
\begin{aligned}
-\Delta Q
+(-\Delta)^{s}Q =
|Q|^{p-2}Q, \ \ x\in\mathbb{R}^{N},
\end{aligned}
\end{equation*}
and
\begin{equation*}
\begin{aligned}
W(Q)=W(u).
\end{aligned}
\end{equation*}
\end{proof}

\begin{lemma}\label{Lemma3.5}
We have
\begin{equation*}
\begin{aligned}
\|Q\|_{D^{1,2}(\mathbb{R}^{N})}^{2}
=
\frac{p(N-2s)-2N}{2N-p(N-2)}
\|Q\|_{D^{s,2}(\mathbb{R}^{N})}^{2}
=
\frac{p(N-2s)-2N}{2p(1-s)}
\|Q\|_{L^{p}(\mathbb{R}^{N})}^{p},
\end{aligned}
\end{equation*}

\begin{equation*}
\begin{aligned}
\frac{N-2}{2}
\|Q\|^{2}_{D^{1,2}(\mathbb{R}^{N})}
+
\frac{N-2s}{2}
\|Q\|^{2}_{D^{s,2}(\mathbb{R}^{N})}
=\frac{N}{p}\int_{\mathbb{R}^{N}}|Q|^{p}\mathrm{d}x,
\end{aligned}
\end{equation*}
and
\begin{equation*}
\begin{aligned}
C_{N,p,s}^{-1}=
\left(
\frac{2N-p(N-2)}{2p(1-s)}
\right)
^{\frac{2N-p(N-2)}{4(1-s)}}
\left(
\frac{p(N-2s)-2N}{2p(1-s)}
\right)
^{\frac{p(N-2s)-2N}{4(1-s)}}
\left(
\int_{\mathbb{R}^{N}}
|Q|^{p}
\mathrm{d}x
\right)^{\frac{p-2}{2}}.
\end{aligned}
\end{equation*}
\end{lemma}

\begin{proof}
From Lemma \ref{Lemma3.4} and $\|u\|_{D^{1,2}(\mathbb{R}^{N})}^{2}=\|u\|_{D^{s,2}(\mathbb{R}^{N})}^{2}=1$, we compute
\begin{equation*}
\begin{aligned}
\frac{
\|Q\|_{D^{1,2}(\mathbb{R}^{N})}^{2}
}
{\|Q\|_{D^{s,2}(\mathbb{R}^{N})}^{2}}
=
\frac{
\lambda_{1}^{2}\lambda_{2}^{2-N}
\|u\|_{D^{1,2}(\mathbb{R}^{N})}^{2}
}
{
\lambda_{1}^{2}\lambda_{2}^{2s-N}
\|u\|_{D^{s,2}(\mathbb{R}^{N})}^{2}}
=
\lambda_{2}^{2-2s}
=
\frac{p(N-2s)-2N}{2N-p(N-2)},
\end{aligned}
\end{equation*}
and
\begin{equation*}
\begin{aligned}
\frac{
\|Q\|_{D^{1,2}(\mathbb{R}^{N})}^{2}
}
{\|Q\|_{L^{p}(\mathbb{R}^{N})}^{p}}
=&
\frac{
\lambda_{1}^{2}\lambda_{2}^{2-N}
\|u\|_{D^{1,2}(\mathbb{R}^{N})}^{2}
}
{
\lambda_{1}^{p}\lambda_{2}^{-N}
\|u\|_{L^{p}(\mathbb{R}^{N})}^{p}}\\
=&
\frac{
\lambda_{1}^{2-p}\lambda_{2}^{2}
}
{\|u\|_{L^{p}(\mathbb{R}^{N})}^{p}}\\
=&
\frac{p(N-2s)-2N}{2p(1-s)}.
\end{aligned}
\end{equation*}
Hence
\begin{equation}\label{3.4}
\begin{aligned}
\|Q\|_{D^{1,2}(\mathbb{R}^{N})}^{2}
=
\frac{p(N-2s)-2N}{2N-p(N-2)}
\|Q\|_{D^{s,2}(\mathbb{R}^{N})}^{2}
=
\frac{p(N-2s)-2N}{2p(1-s)}
\|Q\|_{L^{p}(\mathbb{R}^{N})}^{p}.
\end{aligned}
\end{equation}
Using the fact that \eqref{3.4}, we have
\begin{equation*}\label{3.5}
\begin{aligned}
\frac{N-2}{2}
\|Q\|^{2}_{D^{1,2}(\mathbb{R}^{N})}
+
\frac{N-2s}{2}
\|Q\|^{2}_{D^{s,2}(\mathbb{R}^{N})}
=\frac{N}{p}\int_{\mathbb{R}^{N}}|Q|^{p}\mathrm{d}x.
\end{aligned}
\end{equation*}
Since $Q$ is a weak solution, we have
\begin{equation*}
\begin{aligned}
\|Q\|^{2}_{D^{1,2}(\mathbb{R}^{N})}
+
\|Q\|^{2}_{D^{s,2}(\mathbb{R}^{N})}
=\int_{\mathbb{R}^{N}}|Q|^{p}\mathrm{d}x.
\end{aligned}
\end{equation*}
Then
\begin{equation*}
\begin{aligned}
C_{N,p,s}^{-1}
=&
\frac
{\left(
\|Q\|_{D^{s,2}(\mathbb{R}^{N})}^{2}
\right)
^{\frac{2N-p(N-2)}{4(1-s)}}
\left(
\|Q\|_{D^{1,2}(\mathbb{R}^{N})}^{2}
\right)
^{\frac{p(N-2s)-2N}{4(1-s)}}}
{\int_{\mathbb{R}^{N}}
|Q|^{p}
\mathrm{d}x}\\
=&
\frac
{\left(
\frac{2N-p(N-2)}{2p(1-s)}
\|Q\|_{L^{p}(\mathbb{R}^{N})}^{p}
\right)
^{\frac{2N-p(N-2)}{4(1-s)}}
\left(
\frac{p(N-2s)-2N}{2p(1-s)}
\|Q\|_{L^{p}(\mathbb{R}^{N})}^{p}
\right)
^{\frac{p(N-2s)-2N}{4(1-s)}}}
{\int_{\mathbb{R}^{N}}
|Q|^{p}
\mathrm{d}x}\\
=&
\left(
\frac{2N-p(N-2)}{2p(1-s)}
\right)
^{\frac{2N-p(N-2)}{4(1-s)}}
\left(
\frac{p(N-2s)-2N}{2p(1-s)}
\right)
^{\frac{p(N-2s)-2N}{4(1-s)}}
\left(
\int_{\mathbb{R}^{N}}
|Q|^{p}
\mathrm{d}x
\right)^{\frac{p-2}{2}}.
\end{aligned}
\end{equation*}
\end{proof}

\section{The proof of Theorem \ref{Theorem1.2}}
In this section, we show the existence of ground state solutions for equation \eqref{P} via Mountain-pass geometric structure and Nehari manifold.
The energy functional corresponding to the equation \eqref{P} is
\begin{equation*}
\begin{aligned}
I(u)=
\frac{1}{2}
\left(
\int_{\mathbb{R}^{N}}
|\nabla u|^{2}
\mathrm{d}x
+
\int_{\mathbb{R}^{N}}\int_{\mathbb{R}^{N}}\frac{|u(x)-u(y)|^{2}}{|x-y|^{N+2s}}
\mathrm{d}x\mathrm{d}y\right)
-
\frac{1}{p}
\int_{\mathbb{R}^{N}}
|u|^{p}
\mathrm{d}x,
\end{aligned}
\end{equation*}
Note that $I'(u)$ is the Fr\'{e}chet derivative of $I(u)$, where $\phi\in E$, is given by:
\begin{equation*}
\begin{aligned}
\langle I'(u),\phi\rangle
=&
\int_{\mathbb{R}^{N}}
\nabla u
\nabla \phi
\mathrm{d}x
+
\int_{\mathbb{R}^{N}}
\int_{\mathbb{R}^{N}}
\frac{(u(x)-u(y))(\phi(x)-\phi(y))}{|x-y|^{N+2s}}
\mathrm{d}x
\mathrm{d}y
-
\int_{\mathbb{R}^{N}}
|u|^{p-2}u\phi
\mathrm{d}x.
\end{aligned}
\end{equation*}
We set
\begin{equation*}
\begin{aligned}
c=\inf\limits_{\gamma\in \Gamma}\sup\limits_{t\in[0, 1]}I(\gamma (t))>0\,\ \text{and}\,\   \Gamma=\{\gamma\in C\left([0, 1], E\right)| \gamma(0)=0, I(\gamma (1))<0\}.
\end{aligned}
\end{equation*}
We now set the Nehari manifold as follows
\begin{equation*}
\begin{aligned}
\mathcal{N}=\{u\in E \setminus\{0\}|\langle I'(u),u\rangle=0\}.
\end{aligned}
\end{equation*}

\begin{lemma}\label{Lemma4.1}
Let $N\geqslant3$ and $0<s<1$.
Then we have

\begin{enumerate}
\item [(1)]
the functional $I$ has mountain pass geometric structure;

\item [(2)]
for any $u\in E\setminus\{0\}$, there exists a unique $t_{u}>0$ such that $t_{u}u\in \mathcal{N}$ and $I(t_{u}u)=\max\limits_{t>0}I(tu)$;

\item [(3)]
 $\bar{c}=\inf\limits_{u\in\mathcal{N}}I(u)>0$;

\item [(4)]
$c=\bar{c}=\bar{\bar{c}}$,
where
\begin{equation*}
\begin{aligned}
\bar{\bar{c}}:=\inf\limits_{u\in E\setminus\{0\}}\sup\limits_{t\geqslant 0}I(tu).
\end{aligned}
\end{equation*}

\item [(5)]
For $u\in \mathcal{N}$, we have $\Phi'(u)\not=0$, where
\begin{equation}\label{4.1}
\begin{aligned}
\Phi(u)=\langle I'(u),u\rangle=\|u\|_{E}^{2}-\int_{\mathbb{R}^{N}}|u|^{p}\mathrm{d}x,
\end{aligned}
\end{equation}
and
\begin{equation}\label{4.2}
\begin{aligned}
\langle\Phi'(u),u\rangle
=&2\|u\|_{E}^{2}
-p\beta\int_{\mathbb{R}^{N}}|u|^{p}\mathrm{d}x.
\end{aligned}
\end{equation}
Moreover, if $\bar{u}\in \mathcal{N}$ and $I(\bar{u})=c$, then $u$ is a ground state solution for equation \eqref{P}.
\end{enumerate}
\end{lemma}

The proof details of Lemma \ref{Lemma4.1} is given in Appendix A.
We recall the $(PS)_{c}$ sequence as follows.
\begin{definition}
Let the sequence $\{u_{n}\}\subset E$ satisfy the condition
\begin{equation*}
\begin{aligned}
I(u_{n})\rightarrow c
~\mathrm{and}~I'(u_{n})\rightarrow 0
~\mathrm{in}~E^{-1},
~{\rm as}~n\to\infty.
\end{aligned}
\end{equation*}
Then $\{u_{n}\}$ is called the Palais-Smale sequence of $I$ at level $c$, short for $(PS)_{c}$ sequence,
where $E^{-1}$ is the dual space of $E$.
\end{definition}

\begin{lemma}\label{Lemma4.2}
Let $N\geqslant3$ and $0<s<1$.
Then there exists a bounded $(PS)_{c}$ sequence $\{u_{n}\}\subset \mathcal{N}$ such that
\begin{equation*}
\begin{aligned}
I(u_{n})\rightarrow c
\ \,\mathrm{and}\ \,
\|I'(u_n)\|_{E^{-1}}\rightarrow 0,\ \mathrm{as}\ n\rightarrow \infty.
\end{aligned}
\end{equation*}
\end{lemma}

\begin{proof}
From Lemmas \ref{Lemma4.1} (2) and (4), we know that $\mathcal{N}\not=\emptyset$ and $\inf\limits_{u\in \mathcal{N}}I(u)=\bar{c}=c$.
By the Ekeland's variational principle, there exist $\{u_{n}\}\subset \mathcal{N}$ and
$\lambda_{n}\in \mathbb{R}$ such that
\begin{equation*}
\begin{aligned}
I(u_{n})\rightarrow \bar{c}
~\mathrm{and}~I'(u_{n})-\lambda_{n}\Phi'(u_{n})\rightarrow 0
~\mathrm{in}~E^{-1},
~{\rm as}~n\to\infty.
\end{aligned}
\end{equation*}
So
\begin{equation*}
\begin{aligned}
\bar{c}=I(u_{n})
=I(u_{n})
-\frac{1}{p}\langle I'(u_{n}),u_{n}\rangle
\geqslant\left(\frac{1}{2}-\frac{1}{p}\right)
\|u_{n}\|_{E}^{2},
\end{aligned}
\end{equation*}
which implies that $\{u_{n}\}$ is bounded in $E$.

Taking $n\to\infty$, we have
\begin{equation*}
\begin{aligned}
|\langle I'(u_{n}),u_{n}\rangle-\langle\lambda_{n}\Phi'(u_{n}),u_{n}\rangle|
\leqslant \|I'(u_{n})
-\lambda_{n}\Phi'(u_{n})\|_{E^{-1}}\|u_{n}\|_{E}
\rightarrow 0,
\end{aligned}
\end{equation*}
we have
\begin{equation}\label{4.3}
\begin{aligned}
\langle I'(u_{n}),u_{n}\rangle
-\lambda_{n}\langle \Phi'(u_{n}),u_{n}\rangle\rightarrow 0,
~{\rm as}~n\to\infty.
\end{aligned}
\end{equation}
Note that $\{u_{n}\}\subset \mathcal{N}$. From Lemma \ref{Lemma4.1} (5), we have
\begin{equation}\label{4.4}
\begin{aligned}
\langle I'(u_{n}),u_{n}\rangle=0,
\end{aligned}
\end{equation}
and
\begin{equation}\label{4.5}
\begin{aligned}
\langle \Phi'(u_{n}),u_{n}\rangle\not=0.
\end{aligned}
\end{equation}
Combining \eqref{4.3}-\eqref{4.5},
we conclude that $\lambda_{n}\rightarrow0$.

It follows from H\"{o}lder's and Sobolev's inequalities that
\begin{equation*}
\begin{aligned}
&\|I'(u_{n})\|_{E^{-1}}\\
=&\sup_{\varphi\in E,\|\varphi\|_{E}=1}
|\langle \Phi'(u_{n}),\varphi\rangle|\\
=&\sup_{\varphi\in E,\|\varphi\|_{E}=1}
\left|
2\int_{\mathbb{R}^{N}}
\nabla u \nabla \varphi
\mathrm{d}x
-p
\int_{\mathbb{R}^{N}}
|u_{n}|^{p-2}u_{n}\varphi
\mathrm{d}x
\right.\\
&\left.+2
\int_{\mathbb{R}^{N}}
\int_{\mathbb{R}^{N}}
\frac{(u_{n}(x)-u_{n}(y))(\varphi(x)-\varphi(y))}{|x-y|^{N+2s}}
\mathrm{d}x
\mathrm{d}y
\right|\\
\leqslant&C.
\end{aligned}
\end{equation*}
Then we obtain
\begin{equation*}
\begin{aligned}
\|I'(u_{n})\|_{E^{-1}}
\leqslant\|I'(u_{n})-\lambda_{n}\Phi'(u_{n})\|_{E^{-1}}
+|\lambda_{n}|\|\Phi'(u_{n})\|_{E^{-1}}
=o_{n}(1).
\end{aligned}
\end{equation*}
That is, $I'(u_{n})\rightarrow0$ in $E^{-1}$.
Hence, $\{u_{n}\}$ is a $(PS)_{{c}}$ sequence of $I$.
\end{proof}
We are now in a position to prove Theorem \ref{Theorem1.2}.
\begin{proof}[Proof of Theorem \ref{Theorem1.2}]
From Lemma \ref{Lemma4.2}, we know that there exists a bounded $(PS)_{c}$ sequence $\{u_{n}\}\subset \mathcal{N}$ with $c>0$.
If $\lim\limits_{n\to\infty}\int_{\mathbb{R}^{N}}|u_{n}|^{p}\mathrm{d}x=0$, then
\begin{equation*}
\begin{aligned}
c=I(u_{n})
=\frac{1}{2}\|u_{n}\|_{E}^{2},
\end{aligned}
\end{equation*}
and
\begin{equation*}
\begin{aligned}
0=\langle I'(u_{n}),u_{n}\rangle
=\|u_{n}\|_{E}^{2},
\end{aligned}
\end{equation*}
which gives
\begin{equation*}
\begin{aligned}
c=0.
\end{aligned}
\end{equation*}
Therefore, we get $\lim\limits_{n\to\infty}\int_{\mathbb{R}^{N}}|u_{n}|^{p}\mathrm{d}x>0$.

By using Lemma \ref{Lemma2.5}, there exists $y_{n}\subset \mathbb{R}^{N}$ such that  $\bar{u}_{n}:=u_{n}(x+y_{n})\rightharpoonup u\not \equiv0$ in $L^{p}_{loc}(\mathbb{R}^{N})$ and
\begin{equation*}
\begin{aligned}
c
=I(\bar{u}_{n}),~~\mathrm{and}~~0
=\langle I'(\bar{u}_{n}),\varphi\rangle
=\langle I'(u),\varphi\rangle.
\end{aligned}
\end{equation*}
Now, by virtue of the Brézis-Lieb Lemma \cite{Brezis-Lieb1983PAMS}, one deduces that
\begin{equation*}
\begin{aligned}
\bar{c}
\leqslant I(u)
=&
I(u)-\frac{1}{p}\langle I'(u),u\rangle\\
=&
\left(\frac{1}{2}-\frac{1}{p}\right)
\|u\|_{E}^{2}\\
\leqslant&
\lim\limits_{n\to\infty}
\left(\frac{1}{2}-\frac{1}{p}\right)
\|u_{n}\|_{E}^{2}\\
=&
\lim\limits_{n\to\infty}
I(u_{n})
-\frac{1}{p}
\lim\limits_{n\to\infty}
\langle I'(u_{n}),u_{n}\rangle\\
=&c
=\bar{c}.
\end{aligned}
\end{equation*}
That is,
$I(u)=\bar{c}$.
Therefore,
equation \eqref{P} has a ground state solution.
\end{proof}

\section{The proof of Theorem \ref{Theorem1.3}}
In this section, we prove the existence of solutions in the following set
\begin{equation*}
\begin{aligned}
\mathcal{M}:=\{u\in E\backslash\{0\}|P(u)=0\},
\end{aligned}
\end{equation*}
where
\begin{equation*}
\begin{aligned}
P(u)=\frac{N-2}{2}
\|u\|^{2}_{D^{1,2}(\mathbb{R}^{N})}
+
\frac{N-2s}{2}
\|u\|^{2}_{D^{s,2}(\mathbb{R}^{N})}
-\frac{N}{p}\|u\|_{L^{p}(\mathbb{R}^{N})}^{p}.
\end{aligned}
\end{equation*}

\begin{lemma}\label{Lemma5.1}
Let
$C_{1},C_{2},C_{3}>0$
and
$g_{1}:[0,\infty)\to \mathbb{R}$
as
\begin{equation*}
\begin{aligned}
g_{1}(t)
=
C_{1}t^{N-2}
+
C_{2}t^{N-2s}
-
C_{3}t^{N}.
\end{aligned}
\end{equation*}
Then $g_{1}$ has a unique critical point which corresponds to its maximum.
\end{lemma}

\begin{proof}
We compute
\begin{equation*}
\begin{aligned}
g_{1}'(t)
=
C_{1}(N-2)t^{N-3}
+
C_{2}(N-2s)t^{N-2s-1}
-
C_{3}Nt^{N-1}.
\end{aligned}
\end{equation*}
CLearly, $g_{1}'(t)>0$ for $t>0$ small, $g_{1}'(t)<0$ for $t>0$ large.
Then there exists $\bar{t}>0$ such that  $g_{1}'(\bar{t})=0$.

To prove the uniqueness of $\bar{t}$, we suppose that
$0<\bar{t}<\bar{\bar{t}}$
satisfy
$g_{1}'(\bar{t})=g_{1}'(\bar{\bar{t}})=0$.
Then
\begin{equation*}
\begin{aligned}
\frac{g_{1}'(\bar{t})}{\bar{t}^{N-2s-1}}
-
\frac{g_{1}'(\bar{\bar{t}})}{\bar{\bar{t}}^{N-2s-1}}=0,
\end{aligned}
\end{equation*}
which shows
\begin{equation*}
\begin{aligned}
C_{1}(N-2)
(\bar{t}^{2s-2}-\bar{\bar{t}}^{2s-2})
=
C_{3}N
(\bar{t}^{2s}-\bar{\bar{t}}^{2s}).
\end{aligned}
\end{equation*}
This yields a contradiction.
\end{proof}

\begin{lemma}\label{Lemma5.2}
For all
$u\in E\backslash\{0\}$,
there exists a unique
$t_{u}> 0$
such that
$u(\frac{x}{t_{u}})\in \mathcal{M}$.
\end{lemma}
\begin{proof}
Let
$u=E\setminus\{0\}$ and $t>0$. Set
\begin{equation*}
\begin{aligned}
g_{2}(t)
=&
I\left(u\left(\frac{x}{t}\right)\right)\\
=&
\frac{t^{N-2}}{2}
\|u\|_{D^{1,2}(\mathbb{R}^{N})}^{2}
+
\frac{t^{N-2s}}{2}
\|u\|_{D^{s,2}(\mathbb{R}^{N})}^{2}
-
\frac{t^{N}}{p}
\int_{\mathbb{R}^{N}}
|u|^{p}
\mathrm{d}x.
\end{aligned}
\end{equation*}
Clearly,
\begin{equation*}
\begin{aligned}
g_{2}'(t)
=&
\frac{N-2}{2}t^{N-3}
\|u\|_{D^{1,2}(\mathbb{R}^{N})}^{2}
+
\frac{N-2s}{2}
t^{N-2s-1}
\|u\|_{D^{s,2}(\mathbb{R}^{N})}^{2}
-
\frac{N}{p}t^{N-1}
\int_{\mathbb{R}^{N}}
|u|^{p}
\mathrm{d}x
\end{aligned}
\end{equation*}
and
\begin{equation}\label{5.1}
\begin{aligned}
tg_{2}'(t)=P\left(u\left(\frac{x}{t}\right)\right).
\end{aligned}
\end{equation}
Combining \eqref{5.1} and Lemma \ref{Lemma5.1},
the result is valid.
\end{proof}

\begin{lemma}\label{Lemma5.3}
\begin{enumerate}
\item [(i)]
$\mathcal{M}\not=\emptyset$
and for all
$u\in \mathcal{M}$,
$I(u)=\max\limits_{t>0}I\left(u\left(\frac{x}{t}\right)\right)$;

\item [(ii)]
$\inf\limits_{u\in \mathcal{M}}I(u)=
\inf\limits_{u\in E\setminus\{0\}}\max
\limits_{t\geqslant0}
I\left(u\left(\frac{x}{t}\right)\right)=c>0$, where
\begin{equation*}
\begin{aligned}
c=\inf\limits_{\gamma\in \Gamma}\sup\limits_{t\in[0, 1]}I(\gamma (t)),
\end{aligned}
\end{equation*}
and
\begin{equation*}
\begin{aligned}
\Gamma=\{\gamma\in C\left([0, 1], E\right)| \gamma(0)=0, I(\gamma (1))<0\}.
\end{aligned}
\end{equation*}

\end{enumerate}
	
\end{lemma}

\begin{proof}
(i)
From Lemmas \ref{Lemma5.1} and \ref{Lemma5.2},
we can deduce that.

(ii)
First, we show $\inf\limits_{u\in \mathcal{M}}I(u)>0$.
For $u\in \mathcal{M}$, by Lemma \ref{Lemma2.1},
we have
\begin{equation*}
\begin{aligned}
0
=&
P(u)
=
\frac{N-2}{2}
\|u\|_{D^{1,2}(\mathbb{R}^{N})}^{2}
+
\frac{N-2s}{2}
\|u\|_{D^{s,2}(\mathbb{R}^{N})}^{2}
-
\frac{N}{p}
\int_{\mathbb{R}^{N}}
|u|^{p}
\mathrm{d}x\\
\geqslant&
\frac{N-2}{2}
\|u\|_{E}^{2}
-
\frac{N}{p}
C_{N,p,s}
\|u\|_{E}^{p},
\end{aligned}
\end{equation*}
which gives
\begin{equation*}
\begin{aligned}
\|u\|_{E}
>C>0.
\end{aligned}
\end{equation*}
Then
\begin{equation*}
\begin{aligned}
I(u)
=&I(u)-\frac{1}{N}P(u)\\
=&
\frac{1}{2}
\left(1-\frac{N-2}{N}\right)
\|u\|_{D^{1,2}(\mathbb{R}^{N})}^{2}
+
\frac{1}{2}
\left(1-\frac{N-2s}{N}\right)
\|u\|_{D^{s,2}(\mathbb{R}^{N})}^{2}\\
\geqslant&
\frac{s}{N}
\|u\|_{E}^{2}
\geqslant
C>0.
\end{aligned}
\end{equation*}
This gives
$\inf\limits_{u\in \mathcal{M}}I(u)>0$.

Second, we prove that
$\inf\limits_{u\in \mathcal{M}}I(u)=
\inf\limits_{u\in E\setminus\{0\}}\max
\limits_{t\geqslant0}
I\left(u\left(\frac{x}{t}\right)\right)$.
Clearly,
\begin{equation*}
\begin{aligned}
\inf\limits_{u\in E\setminus\{0\}}\max
\limits_{t\geqslant0}
I\left(u\left(\frac{x}{t}\right)\right)
\leqslant
\inf\limits_{u\in \mathcal{M}}\max
\limits_{t\geqslant0}
I\left(u\left(\frac{x}{t}\right)\right)
=
\inf\limits_{u\in \mathcal{M}}I(u).
\end{aligned}
\end{equation*}
On the other hand, for any
$u\in E\setminus\{0\}$,
by Lemma \ref{Lemma5.2},
we have
\begin{equation*}
\begin{aligned}
\max
\limits_{t\geqslant0}
I\left(u\left(\frac{x}{t}\right)\right)
\geqslant
I\left(u\left(\frac{x}{t_{u}}\right)\right)
\geqslant
\inf\limits_{u\in \mathcal{M}}I(u).
\end{aligned}
\end{equation*}
This implies
\begin{equation*}
\begin{aligned}
\inf\limits_{u\in E\setminus\{0\}}\max
\limits_{t\geqslant0}
J\left(u\left(\frac{x}{t}\right)\right)
\geqslant
\inf\limits_{u\in \mathcal{M}}J(u).
\end{aligned}
\end{equation*}
Thus, the conclusion holds.

Third, we prove that $\inf\limits_{u\in E\setminus\{0\}}\max
\limits_{t\geqslant0}
I\left(u\left(\frac{x}{t}\right)\right)=c$:
For any $u\in E\setminus\{0\}$,
there exists some $\tilde{t}>0$ large,
such that $I(\tilde{t}u)<0$.
Define a path $\gamma:[0, 1]\to E$ by $\gamma(t) = t\tilde{t}u$.
Clearly, $\gamma\in\Gamma$ and
\begin{equation*}
\begin{aligned}
c=
\inf\limits_{\gamma\in \Gamma}\sup\limits_{t\in[0, 1]}I(\gamma (t))
\leqslant
\inf\limits_{u\in E\setminus\{0\}}\max
\limits_{t\geqslant0}
J\left(u\left(\frac{x}{t}\right)\right).
\end{aligned}
\end{equation*}
On the other hand, for each path $\gamma\in\Gamma$,
a direct calculation gives
\begin{equation*}
\begin{aligned}
I(\gamma(1))
-\frac{1}{N}
P(\gamma(1))
\geqslant
\left(\frac{1}{2}-\frac{N-2}{2N}\right)\|\gamma(1)\|_{D^{1,2}(\mathbb{R}^{N})}^{2}
+
\left(\frac{1}{2}-\frac{N-2s}{2N}\right)\|\gamma(1)\|_{D^{s,2}(\mathbb{R}^{N})}^{2}
\geqslant 0,
\end{aligned}
\end{equation*}
which implies
\begin{equation*}
\begin{aligned}
P(\gamma(1))
\leqslant N\cdot I(\gamma(1))
=N\cdot I(\tilde{t}u)<0.
\end{aligned}
\end{equation*}
Note that $P(\gamma(t))>0$ for $t$ small enough.
Then there exists $\tilde{\tilde{t}} \in(0,1)$ such that $P(\gamma(\tilde{\tilde{t}})) = 0$, i.e. $\gamma(\tilde{\tilde{t}})\in \mathcal{M}$ and
\begin{equation*}
\begin{aligned}
c=
\inf\limits_{\gamma\in \Gamma}\sup\limits_{t\in[0, 1]}I(\gamma (t))
\geqslant
\inf\limits_{u\in E\setminus\{0\}}\max
\limits_{t\geqslant0}
J\left(u\left(\frac{x}{t}\right)\right).
\end{aligned}
\end{equation*}
\end{proof}

\begin{lemma}\label{Lemma5.4}
For every $u\in E$,
$t>0$,
\begin{equation*}
\begin{aligned}
J(u)
-
\frac{1}{N}P(u)
-
J\left(u\left(\frac{x}{t}\right)\right)
+
\frac{t^{N}}{N}P(u)
=
g_{3}(t)
\|u\|_{D^{1,2}(\mathbb{R}^{N})}^{2}
+
g_{4}(t)
\|u\|_{D^{s,2}(\mathbb{R}^{N})}^{2},
\end{aligned}
\end{equation*}
where
\begin{equation*}
\begin{aligned}
g_{3}(t)
=&
\frac{1}{N}
+
\frac{N-2}{2N}t^{N}
-
\frac{1}{2}
t^{N-2}
>
g_{3}(1)=0,
\ \,
t\in(0,1)\cup (1,\infty),
\\
g_{4}(t)
=&
\frac{s}{N}
+
+
\frac{N-2s}{2N}t^{N}
-
\frac{1}{2}
t^{N-2s}
>
g_{4}(1)=0,
\ \,
t\in(0,1)\cup(1,\infty).
\end{aligned}
\end{equation*}	
\end{lemma}

\begin{proof}
Clearly,
\begin{equation*}
\begin{aligned}
J(u)-\frac{1}{N}P(u)
=
\frac{1}{N}
\|u\|_{D^{1,2}(\mathbb{R}^{N})}^{2}
+
\frac{s}{N}
\|u\|_{D^{s,2}(\mathbb{R}^{N})}^{2},
\end{aligned}
\end{equation*}
and
\begin{equation*}
\begin{aligned}
&\frac{t^{N}}{N}
P(u)
-
J\left(u\left(\frac{x}{t}\right)\right)\\
=&
\frac{1}{2}
\left(
\frac{N-2}{N}t^{N}
-
t^{N-2}
\right)
\|u\|_{D^{1,2}(\mathbb{R}^{N})}^{2}
+
\frac{1}{2}
\left(
\frac{N-2s}{N}t^{N}
-
t^{N-2s}\right)
\|u\|_{D^{s,2}(\mathbb{R}^{N})}^{2}.
\end{aligned}
\end{equation*}
Combining the above two expressions, we have the desired results.
\end{proof}

\begin{lemma}\label{Lemma5.5}
For all $u\in \mathcal{M}$, we have $P'(u)\not=0$.
\end{lemma}

\begin{proof}
For all $u\in \mathcal{M}$,
we recall some useful identities, $k\geqslant c$,
\begin{equation}\label{5.2}
\begin{aligned}
k=
\frac{1}{2}
\|u\|^{2}_{D^{1,2}(\mathbb{R}^{N})}
+
\frac{1}{2}
\|u\|^{2}_{D^{s,2}(\mathbb{R}^{N})}
-\frac{1}{p}\int_{\mathbb{R}^{N}}|u|^{p}\mathrm{d}x,
\end{aligned}
\end{equation}
and
\begin{equation}\label{5.3}
\begin{aligned}
0=P(u)=
\frac{N-2}{2}
\|u\|^{2}_{D^{1,2}(\mathbb{R}^{N})}
+
\frac{N-2s}{2}
\|u\|^{2}_{D^{s,2}(\mathbb{R}^{N})}
-\frac{N}{p}\int_{\mathbb{R}^{N}}|u|^{p}\mathrm{d}x.
\end{aligned}
\end{equation}
We prove that $P'(u)\not=0$.
Assume on the contrary that
\begin{equation}\label{5.4}
\begin{aligned}
0=\langle P'(u),u\rangle=
(N-2)
\|u\|^{2}_{D^{1,2}(\mathbb{R}^{N})}
+
(N-2s)
\|u\|^{2}_{D^{s,2}(\mathbb{R}^{N})}
-N\int_{\mathbb{R}^{N}}|u|^{p}\mathrm{d}x.
\end{aligned}
\end{equation}
Set
\begin{equation*}
\begin{aligned}
x_{1}=\|u\|^{2}_{D^{1,2}(\mathbb{R}^{N})},
~~x_{2}=\|u\|^{2}_{D^{s,2}(\mathbb{R}^{N})},
x_{3}=\int_{\mathbb{R}^{N}}|u|^{p}\mathrm{d}x.
\end{aligned}
\end{equation*}
Then we have
\begin{equation*}
\begin{aligned}
\begin{cases}
\frac{1}{2}x_{1}+\frac{1}{2}x_{2}-\frac{1}{p}x_{3}=k\\
\frac{N-2}{2}x_{1}+\frac{N-2s}{2}x_{2}-\frac{N}{p}x_{3}=0\\
(N-2)x_{1}+(N-2s)x_{2}-Nx_{3}=0\\
\end{cases}
\Leftrightarrow
D
\begin{bmatrix}
x_{1}\\
x_{2}\\
x_{3}\\
\end{bmatrix}
=
\begin{bmatrix}
k\\
0\\
0\\
\end{bmatrix},
\end{aligned}
\end{equation*}
where $D$ is the coefficient matrix of the system.
In the system, the first equation comes from \eqref{5.2},
the second one holds since \eqref{5.3},
and the third equation is \eqref{5.4}.

It is easy to see that
\begin{equation*}
\begin{aligned}
\mathrm{det} D
=
\left|
\begin{matrix}
\frac{1}{2}&\frac{1}{2}&-\frac{1}{p}\\
\frac{N-2}{2}&\frac{N-2s}{2}&-\frac{N}{p}\\
N-2&N-2s&-N\\
\end{matrix}
\right|
=&-\frac{N(p-2)(1-s)}{2p}
\end{aligned}
\end{equation*}
and
\begin{equation*}
\begin{aligned}
\mathrm{det} D_{1}
=
\left|
\begin{matrix}
k&\frac{1}{2}&-\frac{1}{p}\\
0&\frac{N-2s}{2}&-\frac{N}{p}\\
0&N-2s&-N\\
\end{matrix}
\right|
=&-k\frac{N(p-2)(N-2s)}{2p}
\end{aligned}
\end{equation*}
and
\begin{equation*}
\begin{aligned}
\mathrm{det} D_{2}
=
\left|
\begin{matrix}
\frac{1}{2}&k&-\frac{1}{p}\\
\frac{N-2}{2}&0&-\frac{N}{p}\\
N-2&0&-N\\
\end{matrix}
\right|
=&k\frac{N(p-2)(N-2)}{2p}
\end{aligned}
\end{equation*}
and
\begin{equation*}
\begin{aligned}
\mathrm{det} D_{3}
=
\left|
\begin{matrix}
\frac{1}{2}&\frac{1}{2}&k\\
\frac{N-2}{2}&\frac{N-2s}{2}&0\\
N-2&N-2s&0\\
\end{matrix}
\right|
=
0.
\end{aligned}
\end{equation*}
Then
\begin{equation*}
\begin{aligned}
x_{1}=\frac{\mathrm{det} D_{1}}{\mathrm{det} D}
=
\frac{2p}{N(p-2)(1-s)}\cdot
k\cdot
\frac{N(p-2)(N-2s)}{2p}
=
k\frac{N-2s}{1-s}>0
\end{aligned}
\end{equation*}
and
\begin{equation*}
\begin{aligned}
x_{2}=\frac{\mathrm{det} D_{2}}{\mathrm{det} D}
=
-
\frac{2p}{N(p-2)(1-s)}\cdot
k\cdot
\frac{N(p-2)(N-2)}{2p}
=
-k\frac{N-2}{1-s}<0
\end{aligned}
\end{equation*}
and
\begin{equation*}
\begin{aligned}
x_{3}=\frac{\mathrm{det} D_{1}}{\mathrm{det} D}
=
0
\end{aligned}
\end{equation*}
Note that $N>2s$, $1>s$ and $p>2$. Then $x_{2}<0$ and $x_{3}=0$,  what are not possible.
Hence, one has $P'(u)\not=0$ for all $u\in \mathcal{M}$.
\end{proof}

We recall the $(C)_{c}$  sequence as follows.
\begin{definition}
If sequence $\{u_{n}\}\subset E$ satisfies the condition
\begin{equation*}
\begin{aligned}
I(u_{n})\rightarrow c
~\mathrm{and}~\|I'(u_{n})\|_{E^{-1}}(1+\|u_{n}\|_{E})\rightarrow 0,
~{\rm as}~n\to\infty.
\end{aligned}
\end{equation*}
Then $\{u_{n}\}$ is called the Cerami sequence of $I$ with respect to $c$, short for $(C)_{c}$ sequence.
\end{definition}

We recall a general minimax principle
\cite[Proposition 2.8]{LiGB-WangCH2011AFM},
which is a somewhat stronger variant of \cite[Theorem 2.8]{Willem1996Book}.
\begin{lemma}\label{Lemma5.6}
Let $X$ be a Banach space and $M$ a metric space. Let $M_{0}$ be a closed subspace of $M$ and $\Gamma_{0}\subset C(M_{0},X)$.
Define
\begin{equation*}
\begin{aligned}
\Gamma:=\{\gamma\in C(M,X):\gamma|_{M_{0}}\in\Gamma_{0}\}.
\end{aligned}
\end{equation*}
If $\varphi\in C^{1}(X,\mathbb{R})$ satisfies
\begin{equation*}
\begin{aligned}
a:=\sup\limits_{\gamma_{0}\in \Gamma_{0}}\sup\limits_{u\in M_{0}}I(\gamma (t))
<
c:=\inf\limits_{\gamma\in \Gamma}\sup\limits_{u\in M}I(\gamma (t))<\infty,
\end{aligned}
\end{equation*}
then, for every $\varepsilon\in(0,\frac{c-a}{2})$,
$\delta>0$
and
$\gamma\in\Gamma$ such that
\begin{equation*}
\begin{aligned}
\sup\limits_{M}\varphi\circ\gamma\leqslant c+\varepsilon,
\end{aligned}
\end{equation*}
there exists $u\in X$ such that
\begin{enumerate}
\item [(i)] $c-2\varepsilon \leqslant \varphi(u)\leqslant c+2\varepsilon$;

\item [(ii)]
$\mathrm{dist}(u,\gamma(M))\leqslant 2\delta$;

\item [(iii)]
$(1+\|u_{n}\|_{X})\|\varphi'(u)\|_{X^{-1}}<\frac{8\varepsilon}{\delta}$, where $X^{-1}$ is the dual space of $X$.
\end{enumerate}
\end{lemma}
In the following, we apply Lemma \ref{Lemma5.6} to obtain a Cerami sequence for the functional $I$ with $P(u_{n})\to0$.

\begin{lemma}\label{Lemma5.7}
Let $N\geqslant3$ and $0<s<1$.
Then there exists a bounded sequence $\{u_{n}\}\subset E$ such that
\begin{equation*}
\begin{aligned}
I(u_{n})\rightarrow c,\ \
P(u_{n})\rightarrow 0,
\ \,\mathrm{and}\ \,
\|I'(u_{n})\|_{E^{-1}}(1+\|u_{n}\|_{E})\rightarrow 0,\ \mathrm{as}\ n\rightarrow \infty.
\end{aligned}
\end{equation*}
\end{lemma}

\begin{proof}
Let us define the continuous map $h:\mathbb{R}\times E\to E$ as
\begin{equation*}
\begin{aligned}
h(z,v)=v(e^{-z}x),~~z\in\mathbb{R},~~v\in E,~~x\in\mathbb{R}^{N},
\end{aligned}
\end{equation*}
where $\mathbb{R}\times E$ is a Banach space, with the product norm $\|(z,v)\|_{\mathbb{R}\times E}^{2}=|z|^{2}+\|v\|_{E}^{2}$.
We consider the auxiliary functional
\begin{equation*}
\begin{aligned}
J(z,v)=I(h(z,v))=
\frac{e^{(N-2)z}}{2}
\|v\|^{2}_{D^{1,2}(\mathbb{R}^{N})}
+
\frac{e^{(N-2s)z}}{2}
\|v\|^{2}_{D^{s,2}(\mathbb{R}^{N})}
-\frac{e^{Nz}}{p}\int_{\mathbb{R}^{N}}|v|^{p}\mathrm{d}x.
\end{aligned}
\end{equation*}
Clearly,
\begin{equation*}
\begin{aligned}
\frac{d}{\mathrm{d}z}J(z,v)=P(h(z,v))
\end{aligned}
\end{equation*}
and
\begin{equation*}
\begin{aligned}
\langle\frac{d}{\mathrm{d}v}J(z_{1},v),(z_{2},v)\rangle
=
\langle\frac{d}{\mathrm{d}v}I(h(z_{1},v)),h(z_{2},w)\rangle
=
\langle I'(h(z_{1},v)),h(z_{2},w)\rangle
\end{aligned}
\end{equation*}
and
\begin{equation}\label{5.5}
\begin{aligned}
\langle J'(z_{1},v),(z_{2},w)\rangle
=&
\left.\frac{\mathrm{d}}{\mathrm{d}t}\right|_{t=0}J((z_{1}+tz_{2},v+tw))\\
=&
\left.\frac{\mathrm{d}}{\mathrm{d}t}\right|_{t=0}I(h(z_{1}+tz_{2},v+tw))\\
=&
\langle\frac{d}{\mathrm{d}v}I(h(z_{1},v)),h(z_{2},w)\rangle
+
P(h(z_{1},v))z_{2}\\
=&
\langle I'(h(z_{1},v)),h(z_{2},w)\rangle
+
P(h(z_{1},v))z_{2}.
\end{aligned}
\end{equation}
for all $z,z_{1},z_{2}\in \mathbb{R}$ and $v,w\in E$.

Set the minimax value
\begin{equation*}
\begin{aligned}
\tilde{c}:=\inf\limits_{\tilde{\gamma}\in \tilde{\Gamma}}\sup\limits_{t\in [0,1]}J(\tilde{\gamma}(t)),
\end{aligned}
\end{equation*}
where
\begin{equation*}
\begin{aligned}
\tilde{\Gamma}=\{\tilde{\gamma}\in C\left([0, 1], \mathbb{R}\times E\right)| \tilde{\gamma}(0)=(0,0), J(\tilde{\gamma}(1))<0\}.
\end{aligned}
\end{equation*}
Note that $\Gamma:=\{h\circ \tilde{\gamma}:\tilde{\gamma}\in \tilde{\Gamma}\}$. Then
\begin{equation*}
\begin{aligned}
c=\tilde{c}.
\end{aligned}
\end{equation*}
By the definition of $c$, for every $n\in \mathbb{N}$, there exists $\gamma_{n}\in \Gamma$ such that
\begin{equation*}
\begin{aligned}
\sup\limits_{t\in [0,1]}J(0,\gamma(t))
=
\sup\limits_{t\in [0,1]}I(\gamma(t))
\leqslant c+\frac{1}{n^{2}}.
\end{aligned}
\end{equation*}
We apply Lemma \ref{Lemma5.6} with $X=\mathbb{R}\times E$, $M=[0,1]$, $M_{0}=\{0,1\}$ and $\varphi=J$.
Let $\varepsilon_{n}=\frac{1}{n^{2}}$, $\delta_{n}=\frac{1}{n}$ and $\tilde{\gamma}(t)=(0,\gamma(t))$.
Clearly, $\varepsilon_{n}=\frac{1}{n^{2}}\in(0,\frac{c}{2})$ for $n$ large.
Then Lemma \ref{Lemma5.6} yields the existence of $(z_{n},v_{n})\in \mathbb{R}\times E$ such that
\begin{equation*}
\begin{aligned}
(i)&J(z_{n},v_{n})\to c,\\
(ii)&\|J'(z_{n},v_{n})\|_{(\mathbb{R}\times E)^{-1}}(1+\|(z_{n},v_{n})\|_{\mathbb{R}\times E})< \frac{8}{n},\\
(iii)&\mathrm{dist}((z_{n},v_{n}),\{0\}\times \gamma_{n}([0,1]))\leqslant \frac{2}{n}.
\end{aligned}
\end{equation*}
By using (iii), one has  $z_{n}\to0$ as
\begin{equation*}
\begin{aligned}
|z_{n}|=|z_{n}-0|\leqslant \mathrm{dist}((z_{n},v_{n}),\{0\}\times \gamma_{n}([0,1]))\leqslant \frac{2}{n}.
\end{aligned}
\end{equation*}
From (ii) we know
\begin{equation*}
\begin{aligned}
\|J'(z_{n},v_{n})\|_{(\mathbb{R}\times E)^{-1}}< \frac{8}{n}
\end{aligned}
\end{equation*}
Furthermore, taking $z_{2}=1$ and $w=0$ in \eqref{5.5}, we have
\begin{equation*}
\begin{aligned}
P(h(z_{n},v_{n}))\to 0, ~~\mathrm{as}~~n\to\infty.
\end{aligned}
\end{equation*}
Let $u_{n}:=h(z_{n},v_{n})$.
We have $I(u_{n})\to c$ and $P(u_{n})\to 0$ as $n\to \infty$. We now claim that
\begin{equation*}
\begin{aligned}
\|I'(u_{n})\|_{E^{-1}}(1+\|u_{n}\|_{E})\to 0,~~\mathrm{as}~~n\to\infty.
\end{aligned}
\end{equation*}
For any $\phi\in E$, we set $\tilde{\phi}=h(-z_{n},\phi)$. Then
\begin{equation*}
\begin{aligned}
\|(0,\tilde{\phi})\|_{\mathbb{R}\times E}^{2}
=\|\tilde{\phi}\|_{E}^{2}
=e^{-(N-2)z_{n}}\|\phi\|^{2}_{D^{1,2}(\mathbb{R}^{N})}
+e^{-(N-2s)z_{n}}\|\phi\|^{2}_{D^{s,2}(\mathbb{R}^{N})}
\leqslant C\|\phi\|_{E}^{2}.
\end{aligned}
\end{equation*}
Then
\begin{equation*}
\begin{aligned}
|
\langle
I'(u_{n}),\phi
\rangle
|
(1+\|u_{n}\|_{E})
=&
|
\langle
J'(z_{n},v_{n}),(0,\tilde{\phi})
\rangle
|
(1+\|(z_{n},v_{n})\|_{\mathbb{R}\times E})\\
\leqslant&
\|J'(z_{n},v_{n})\|_{(\mathbb{R}\times E)^{-1}}
\|(0,\tilde{\phi})\|_{\mathbb{R}\times E}
(1+\|(z_{n},v_{n})\|_{\mathbb{R}\times E})\\
<&\frac{8}{n}C\|\phi\|_{E}^{2}.
\end{aligned}
\end{equation*}

\end{proof}

\begin{lemma}\label{Lemma5.8}
Let $\{u_{n}\}$
be a sequence of
$c>0$ as follows:
\begin{equation*}
\begin{aligned}
I(u_{n})\to c,~~
P(u_{n})\to 0,~~\mathrm{and}~~ \|I'(u_{n})\|_{E^{-1}}(1+\|u_{n}\|_{E})\to 0,\ \mathrm{as}\ n\rightarrow \infty.
\end{aligned}
\end{equation*}
Then it is a bounded sequence in $E$.
If
$u_{n}\rightharpoonup u$
in
$E$,
then
$\|I'(u)\|_{E^{-1}}=0$ and $P(u)=0$. Moreover, if $u\not\equiv0$, then $I(u)=c$.
\end{lemma}

\begin{proof}
From $\|I'(u_{n})\|_{E^{-1}}(1+\|u_{n}\|_{E})\to 0$, we have
$\|I'(u_{n})\|_{E^{-1}}\to 0$ as $n\to\infty$.
Then
\begin{equation*}
\begin{aligned}
c+o_{n}(1)=&
\frac{1}{2}\|u_{n}\|_{D^{1,2}(\mathbb{R}^{N})}^{2}
+\frac{1}{2}\|u_{n}\|_{D^{s,2}(\mathbb{R}^{N})}^{2}
-\frac{1}{p}\|u_{n}\|_{L^{p}(\mathbb{R}^{N})}^{p}\\
&-\frac{1}{p}
\left(\|u_{n}\|_{D^{1,2}(\mathbb{R}^{N})}^{2}
+\|u_{n}\|_{D^{s,2}(\mathbb{R}^{N})}^{2}
-\|u_{n}\|_{L^{p}(\mathbb{R}^{N})}^{p}\right)\\
=&
\frac{p-2}{2p}\|u_{n}\|_{E}^{2}.
\end{aligned}
\end{equation*}
This shows that $\{u_{n}\}$ is bounded in $E$.
Up to a subsequence,
there exists
$u$
such that
\begin{equation*}
\begin{aligned}
u_{n}\rightharpoonup u
\;
\mathrm{in}
~
E,
\;\;
u_{n}\rightarrow u
~
\mathrm{a.e. ~in}
~
\mathbb{R}^{N}.
\end{aligned}
\end{equation*}
From $\|I'(u_{n})\|_{E^{-1}}\to 0$, we have
$\|I'(u)\|_{E^{-1}}=0$.

For
$u\equiv0$,
we have
$P(u)=0$.

From now on, we suppose $u\not\equiv0$.
Firstly, we show
$P(u)\geqslant0$.
By the way of contradiction, we suppose that
$P(u)<0$.
In terms of Lemma
\ref{Lemma5.2},
there exists a unique
$\bar{t}>0$
such that
\begin{equation*}
\begin{aligned}
u\left(\frac{x}{\bar{t}}\right)\in\mathcal{M}
~~\mathrm{and}~~
P\left(u\left(\frac{x}{\bar{t}}\right)\right)=0.
\end{aligned}
\end{equation*}
In view of
\eqref{5.1}
and
$P(u)<0$,
we obtain
\begin{equation*}
\begin{aligned}
g_{2}'(1)=P(u)<0.
\end{aligned}
\end{equation*}
By the expression of $g_{2}'(\cdot)$,
one can choose
$t_{0}>0$ small such that
\begin{equation*}
g_{2}'(t_{0})>0.
\end{equation*}
Using the intermediate value theorem, we know there exists
$\bar{\bar{t}}\in(t_{0},1)$
such that
\begin{equation*}
g_{2}'(\bar{\bar{t}})=0.
\end{equation*}
By the uniqueness of
$\bar{t}$,
one has
\begin{equation}\label{5.6}
\begin{aligned}
\bar{t}
=
\bar{\bar{t}}
\in(t_{0},1).
\end{aligned}
\end{equation}
Making use of \eqref{5.6},
there holds
\begin{equation*}
\begin{aligned}
c
\leqslant&
I\left(u\left(\frac{x}{\bar{t}}\right)\right)
=
I\left(u\left(\frac{x}{\bar{t}}\right)\right)
-
\frac{1}{N}
P\left(u\left(\frac{x}{\bar{t}}\right)\right)\\
=&
\frac{1}{N}
\bar{t}^{N-2}
\|u\|_{D^{1,2}(\mathbb{R}^{N})}^{2}
+
\frac{s}{N}
\bar{t}^{N-2s}
\|u\|_{D^{s,2}(\mathbb{R}^{N})}^{2}
\\
<&
\frac{1}{N}
\|u\|_{D^{1,2}(\mathbb{R}^{N})}^{2}
+
\frac{s}{N}
\|u\|_{D^{s,2}(\mathbb{R}^{N})}^{2}\\
\leqslant&
\frac{1}{N}
\lim_{n\rightarrow\infty}
\|u_{n}\|_{D^{1,2}(\mathbb{R}^{N})}^{2}
+
\frac{s}{N}
\lim_{n\rightarrow\infty}
\|u_{n}\|_{D^{s,2}(\mathbb{R}^{N})}^{2}\\
=&
\lim_{n\rightarrow\infty}
\bigg[
I(u_{n})
-
\frac{1}{N}
P(u_{n})
\bigg]\\
=&
\lim_{n\rightarrow\infty}
I(u_{n})
=
c.
\end{aligned}
\end{equation*}
This is a contradiction.
Hence,
$P(u)\geqslant0$.

We now prove that
$P(u)=0$.
Suppose that
$P(u)>0$.
From
$u\not\equiv0$,
we know
\begin{equation}\label{5.7}
\begin{aligned}
I(u)
-
\frac{1}{N}
P(u)>0.
\end{aligned}
\end{equation}
Set $v_{n}=u_{n}-u$.
According to Br\'{e}zis-Lieb lemma
\cite{Brezis-Lieb1983PAMS},
we know
\begin{equation*}
c
=
\lim_{n\rightarrow\infty} I(u_{n})=\lim_{n\rightarrow\infty} I(v_{n})+I(u)
\end{equation*}
and
\begin{equation*}
0=\lim_{n\rightarrow\infty}P(u_{n})=\lim_{n\rightarrow\infty}P(v_{n})+P(u),
\end{equation*}
which further implies that:
\begin{equation}\label{5.8}
\begin{aligned}
c=&
\lim_{n\rightarrow\infty}
I(u_{n})-
\frac{1}{N}
\lim_{n\rightarrow\infty}P(u_{n})\\
=&
\lim_{n\rightarrow\infty}
\left[
I(v_{n})
-
\frac{1}{N}
P(v_{n})
\right]
+
\left[
I(u)-\frac{1}{N}P(u)
\right]
\end{aligned}
\end{equation}
and
\begin{equation*}
\lim_{n\rightarrow\infty}
P(v_{n})=-P(u)<0.
\end{equation*}
Similar to \eqref{5.6}, there exists a unique sequence $\bar{t}_{n}$ such that
\begin{equation}\label{5.9}
v_{n}\left(\frac{x}{\bar{t}_{n}}\right)\in \mathcal{M}
~~\mathrm{and}~~
P\left(v_{n}\left(\frac{x}{\bar{t}_{n}}\right)\right)=0
\ \,
\mathrm{and}
\ \,
\bar{t}_{n}\in(0,1].
\end{equation}
From
\eqref{5.7}-\eqref{5.9} and $\bar{t}_{n}\in(0,1]$,
we obtain
\begin{equation*}
\begin{aligned}
c
\leqslant&
\lim_{n\rightarrow\infty}
I\left(v_{n}\left(\frac{x}{\bar{t}_{n}}\right)\right)\\
=&
\lim_{n\rightarrow\infty}
\left[
I\left(v_{n}\left(\frac{x}{\bar{t}_{n}}\right)\right)
-
\frac{1}{N}
P\left(v_{n}\left(\frac{x}{\bar{t}_{n}}\right)\right)
\right]\\
<&
\lim_{n\rightarrow\infty}
\left[
I(v_{n})
-
\frac{1}{N}
P(v_{n})
\right]\\
<&
\lim_{n\rightarrow\infty}
\left[
I(v_{n})
-
\frac{1}{N}
P(v_{n})
\right]
+
\left[
I(u)-\frac{1}{N}P(u)
\right]\\
=&
\lim_{n\rightarrow\infty}
I(u_{n})
-
\frac{1}{N}
\lim_{n\rightarrow\infty}
P(u_{n})
=
c.
\end{aligned}
\end{equation*}
This is a contradiction. Hence, $P(u)=0$.

Moreover, if $u\not\equiv0$, then
\begin{equation*}
\begin{aligned}
c
\leqslant
I(u)
=I(u)
-\frac{1}{N}P(u)
\leqslant
\lim_{n\rightarrow\infty}I(u_{n})
-\frac{1}{N}
\lim_{n\rightarrow\infty}
P(u_{n})
=c,
\end{aligned}
\end{equation*}
which implies
$I(u)=c$.
\end{proof}

\begin{proof}[The proof of Theorem \ref{Theorem1.3}]
Let
$\{u_{n}\}$
be a bounded sequence of
$c>0$ which is in Lemma \ref{Lemma5.7}.
If $\lim\limits_{n\to\infty}\int_{\mathbb{R}^{N}}
|u_{n}|^{p}
\mathrm{d}x=0$,
then
\begin{equation*}
\begin{aligned}
c=\frac{1}{2}\lim_{n\rightarrow\infty}\|u_{n}\|_{D^{1,2}(\mathbb{R}^{N})}^{2}
+
\frac{1}{2}\lim_{n\rightarrow\infty}\|u_{n}\|_{D^{s,2}(\mathbb{R}^{N})}^{2},
\end{aligned}
\end{equation*}
and
\begin{equation*}
\begin{aligned}
0=\frac{N-2}{2}
\lim_{n\rightarrow\infty}
\|u_{n}\|_{D^{1,2}(\mathbb{R}^{N})}^{2}
+
\frac{N-2s}{2}
\lim_{n\rightarrow\infty}
\|u_{n}\|_{D^{s,2}(\mathbb{R}^{N})}^{2},
\end{aligned}
\end{equation*}
which further gives
\begin{equation*}
\begin{aligned}
0=\lim_{n\rightarrow\infty}\|u_{n}\|_{D^{1,2}(\mathbb{R}^{N})}^{2},
\end{aligned}
\end{equation*}
\begin{equation*}
\begin{aligned}
0=
\lim_{n\rightarrow\infty}
\|u_{n}\|_{D^{s,2}(\mathbb{R}^{N})}^{2},
\end{aligned}
\end{equation*}
and
\begin{equation*}
\begin{aligned}
c=0.
\end{aligned}
\end{equation*}
This is a contradiction.
Hence,
$\lim\limits_{n\to\infty}\int_{\mathbb{R}^{N}}
|u_{n}|^{p}
\mathrm{d}x>0$.

By Lemma \ref{Lemma2.5} and $\lim\limits_{n\to\infty}\int_{\mathbb{R}^{N}}
|u_{n}|^{p}
\mathrm{d}x>0$,
we know that there exists $\{x_{n}\}\subset \mathbb{R}^{N}$ such that $\{\bar{u}_{n}:=u_{n}(x+x_{n})\}$ convergence strongly and a.e. to $\bar{u}\not\equiv0$ in $L^{p}_{loc}(\mathbb{R}^{N})$.
Clearly, $\{\bar{u}_{n}\}$ satisfies
\begin{equation*}
\begin{aligned}
I(\bar{u}_{n})\to c,~~
P(\bar{u}_{n})\to 0,~~\mathrm{and}~~ \|I'(\bar{u}_{n})\|_{E^{-1}}(1+\|\bar{u}_{n}\|_{E})\to 0,\ \mathrm{as}\ n\rightarrow \infty.
\end{aligned}
\end{equation*}
Up to a subsequence,
there exists
$u$
such that
\begin{equation*}
\begin{aligned}
\bar{u}_{n}\rightharpoonup \bar{u}
\;
\mathrm{in}
~
E,
\;\;
\bar{u}_{n}\rightarrow \bar{u}
~
\mathrm{a.e. ~in}
~
\mathbb{R}^{N}.
\end{aligned}
\end{equation*}
Moreover, from Lemma \ref{Lemma5.8}, we get
\begin{equation*}
\begin{aligned}
I(\bar{u})=c,~~P(\bar{u})=0~\mathrm{and}~\|I'(\bar{u})\|_{E^{-1}}=0.
\end{aligned}
\end{equation*}
\end{proof}

\section{Proof of Theorem \ref{Theorem1.4}}
In this section, we show the best constant based on Theorems \ref{Theorem1.1}-\ref{Theorem1.3}.

From Theorem \ref{Theorem1.3},
we know that the ground sate solutions of equation \eqref{P} satisfy Pohov{z}aev identity.
Hence, equation \eqref{P} has a ground sate solution $\phi$, which satisfies $P(\phi)=0$ and $I(\phi)=c$.
\begin{lemma}\label{Lemma6.1}
Let $u$ be a weak solution of equation \eqref{P} with $P(u)=0$ and $I(u)\geqslant c$.
Then
\begin{equation}\label{6.1}
\begin{aligned}
\|u\|_{D^{1,2}(\mathbb{R}^{N})}^{2}
=
\frac{p(N-2s)-2N}{2N-p(N-2)}
\|u\|_{D^{s,2}(\mathbb{R}^{N})}^{2}
=
\frac{p(N-2s)-2N}{2p(1-s)}
\|u\|_{L^{p}(\mathbb{R}^{N})}^{p},
\end{aligned}
\end{equation}
and
\begin{equation*}
\begin{aligned}
I(u)=\frac{p-2}{2p}
\|u\|_{L^{p}(\mathbb{R}^{N})}^{p}.
\end{aligned}
\end{equation*}
Moreover, let $\phi$ be a weak solution of equation \eqref{P} with $P(\phi)=0$ and $I(\phi)=c$. Then
\begin{equation*}
\begin{aligned}
\|\phi\|_{L^{p}(\mathbb{R}^{N})}^{p}
\leqslant
\|u\|_{L^{p}(\mathbb{R}^{N})}^{p}.
\end{aligned}
\end{equation*}
\end{lemma}
\begin{proof}
It follows from $\langle I'(u),u\rangle=0$ and $P(u)=0$ that
\begin{equation*}
\begin{aligned}
\|u\|_{D^{1,2}(\mathbb{R}^{N})}^{2}
+
\|u\|_{D^{s,2}(\mathbb{R}^{N})}^{2}
-
\|u\|_{L^{p}(\mathbb{R}^{N})}^{p}
=0
\end{aligned}
\end{equation*}
and
\begin{equation*}
\begin{aligned}
\frac{N-2}{2}
\|u\|^{2}_{D^{1,2}(\mathbb{R}^{N})}
+
\frac{N-2s}{2}
\|u\|^{2}_{D^{s,2}(\mathbb{R}^{N})}
=\frac{N}{p}\|u\|_{L^{p}(\mathbb{R}^{N})}^{p}.
\end{aligned}
\end{equation*}
Then
\begin{equation}\label{6.2}
\begin{aligned}
\left(
\frac{N-2}{2}
-\frac{N}{p}
\right)
\|u\|^{2}_{D^{1,2}(\mathbb{R}^{N})}
+
\left(
\frac{N-2s}{2}
-\frac{N}{p}
\right)
\|u\|^{2}_{D^{s,2}(\mathbb{R}^{N})}
=0.
\end{aligned}
\end{equation}
This is the first equation of \eqref{6.1}.
Putting this result back, we get the second equation of \eqref{6.1}.
Furthermore, one deduces
\begin{equation*}
\begin{aligned}
I(u)
=&
\frac{1}{2}
\|u\|_{D^{1,2}(\mathbb{R}^{N})}^{2}
+
\frac{1}{2}
\|u\|_{D^{s,2}(\mathbb{R}^{N})}^{2}
-
\frac{1}{p}
\|u\|_{L^{p}(\mathbb{R}^{N})}^{p}\\
=&
\frac{p-2}{2p}
\int_{\mathbb{R}^{N}}
|u|^{p}
\mathrm{d}x.
\end{aligned}
\end{equation*}
Then
\begin{equation*}
\begin{aligned}
\frac{p-2}{2p}
\|\phi\|_{L^{p}(\mathbb{R}^{N})}^{p}
=
I(\phi)=c
\leqslant
I(u)
=
\frac{p-2}{2p}
\|u\|_{L^{p}(\mathbb{R}^{N})}^{p},
\end{aligned}
\end{equation*}
which gives
\begin{equation*}
\begin{aligned}
\|\phi\|_{L^{p}(\mathbb{R}^{N})}^{p}
\leqslant
\|u\|_{L^{p}(\mathbb{R}^{N})}^{p}.
\end{aligned}
\end{equation*}
\end{proof}

\begin{lemma}\label{Lemma6.2}
Let $\phi$ be a weak solution of equation \eqref{P} with $P(\phi)=0$ and $I(\phi)=c$.
Then
\begin{equation*}
\begin{aligned}
C_{N,p,s}^{-1}\leqslant
\left(
\frac{2N-p(N-2)}{2p(1-s)}
\right)
^{\frac{2N-p(N-2)}{4(1-s)}}
\left(
\frac{p(N-2s)-2N}{2p(1-s)}
\right)
^{\frac{p(N-2s)-2N}{4(1-s)}}
\|\phi\|_{L^{p}(\mathbb{R}^{N})}^{\frac{p(p-2)}{p}}.
\end{aligned}
\end{equation*}
\end{lemma}
\begin{proof}
It follows from \eqref{6.1} that
\begin{equation*}
\begin{aligned}
C_{N,p,s}^{-1}\leqslant& W(\phi)\\
=&
\frac
{\left(
\|\phi\|_{D^{s,2}(\mathbb{R}^{N})}^{2}
\right)
^{\frac{2N-p(N-2)}{4(1-s)}}
\left(
\|\phi\|_{D^{1,2}(\mathbb{R}^{N})}^{2}
\right)
^{\frac{p(N-2s)-2N}{4(1-s)}}}
{\int_{\mathbb{R}^{N}}
|\phi|^{p}
\mathrm{d}x}\\
=&\left(
\frac{2N-p(N-2)}{2p(1-s)}
\right)
^{\frac{2N-p(N-2)}{4(1-s)}}
\left(
\frac{p(N-2s)-2N}{2p(1-s)}
\right)
^{\frac{p(N-2s)-2N}{4(1-s)}}
\|\phi\|_{L^{p}(\mathbb{R}^{N})}^{\frac{p(p-2)}{p}}.
\end{aligned}
\end{equation*}
\end{proof}

\begin{proof}[Proof of Theorem \ref{Theorem1.4}]
It follows from Theorem \ref{Theorem1.1} that
\begin{equation*}
\begin{aligned}
I(Q)\geqslant c,
~~\langle I'(Q),Q\rangle=0, ~~\mathrm{and}~~P(Q)=0.
\end{aligned}
\end{equation*}
From  Theorem \ref{Theorem1.1} and Lemma \ref{Lemma6.1} that
\begin{equation*}
\begin{aligned}
C_{N,p,s}^{-1}
=&
\left(
\frac{2N-p(N-2)}{2p(1-s)}
\right)
^{\frac{2N-p(N-2)}{4(1-s)}}
\left(
\frac{p(N-2s)-2N}{2p(1-s)}
\right)
^{\frac{p(N-2s)-2N}{4(1-s)}}
\|Q\|_{L^{p}(\mathbb{R}^{N})}^{\frac{p(p-2)}{p}}\\
\geqslant&
\left(
\frac{2N-p(N-2)}{2p(1-s)}
\right)
^{\frac{2N-p(N-2)}{4(1-s)}}
\left(
\frac{p(N-2s)-2N}{2p(1-s)}
\right)
^{\frac{p(N-2s)-2N}{4(1-s)}}
\|\phi\|_{L^{p}(\mathbb{R}^{N})}^{\frac{p(p-2)}{p}}.
\end{aligned}
\end{equation*}
On the other hand, by using Lemma \ref{Lemma6.2}, one has
\begin{equation*}
\begin{aligned}
C_{N,p,s}^{-1}
\leqslant
\left(
\frac{2N-p(N-2)}{2p(1-s)}
\right)
^{\frac{2N-p(N-2)}{4(1-s)}}
\left(
\frac{p(N-2s)-2N}{2p(1-s)}
\right)
^{\frac{p(N-2s)-2N}{4(1-s)}}
\|\phi\|_{L^{p}(\mathbb{R}^{N})}^{\frac{p(p-2)}{p}}.
\end{aligned}
\end{equation*}
Hence $\|Q\|_{L^{p}(\mathbb{R}^{N})}=\|\phi\|_{L^{p}(\mathbb{R}^{N})}$, and
\begin{equation*}
\begin{aligned}
C_{N,p,s}^{-1}
=
\left(
\frac{2N-p(N-2)}{2p(1-s)}
\right)
^{\frac{2N-p(N-2)}{4(1-s)}}
\left(
\frac{p(N-2s)-2N}{2p(1-s)}
\right)
^{\frac{p(N-2s)-2N}{4(1-s)}}
\|\phi\|_{L^{p}(\mathbb{R}^{N})}^{\frac{p(p-2)}{p}}.
\end{aligned}
\end{equation*}
Moreover, one deduces that $I(Q)=c$.

Applying Lemma \ref{Lemma6.1} again, we know
\begin{equation*}
\begin{aligned}
c
=&I(\phi)\\
=&\frac{p-2}{2p}\int_{\mathbb{R}^{N}}|\phi|^{p}\mathrm{d}x,
\end{aligned}
\end{equation*}
which shows that
\begin{equation*}
\begin{aligned}
\|\phi\|_{L^{p}(\mathbb{R}^{N})}
=\left[
\frac{2p}{p-2}
c
\right]^{\frac{1}{p}}.
\end{aligned}
\end{equation*}
Then
\begin{equation*}
\begin{aligned}
C_{N,p,s}^{-1}
=\left(
\frac{2N-p(N-2)}{2p(1-s)}
\right)
^{\frac{2N-p(N-2)}{4(1-s)}}
\left(
\frac{p(N-2s)-2N}{2p(1-s)}
\right)
^{\frac{p(N-2s)-2N}{4(1-s)}}
\left[
\frac{2p}{p-2}
c
\right]^{\frac{p-2}{2}}.
\end{aligned}
\end{equation*}
\end{proof}

\section*{Appendix A}
In this appendix, we prove the details of Lemma \ref{Lemma4.1}.
\begin{proof}[Proof of Lemma \ref{Lemma4.1}]
(1) Using Lemma \ref{Lemma2.2}, one has
\begin{equation*}
\begin{aligned}
I(u)
\geqslant \|u\|_{E}^{2}
-C\|u\|_{E}^{p}.
\end{aligned}
\end{equation*}
We keep in mind that $2<p$. Then there exists a sufficiently small positive number $\rho$ such that
\begin{equation*}
\begin{aligned}
\varsigma :=
\inf_{\|u\|_{E}=\rho}
I(u)>0=I(0).
\end{aligned}
\end{equation*}
For $u\in E\setminus\{0\}$, we have
\begin{equation*}
\begin{aligned}
I(tu)
=
\frac{t^{2}}{2}
\|u\|_{E}^{2}
-\frac{t^{p}}{p}
\int_{\mathbb{R}^{N}}|u|^{p}\mathrm{d}x.
\end{aligned}
\end{equation*}
From $2<p$, it follows that $I(tu)<0$ for $t$ large enough.

From above,
we can choose $t_{1}>0$ corresponding to $u$ such that $I(t_{1}u)<0$ for $t>t_{1}$ and $\|t_{1}u\|_{E}>\rho$.

(2)
For any $u\in E\setminus\{0\}$ and $t\in (0,\infty)$, we define
\begin{equation*}
\begin{aligned}
f_{1}(t)=I(tu)=&\frac{t^{2}}{2}\|u\|_{E}^{2}-\frac{t^{p}}{p}\int_{\mathbb{R}^{N}}|u|^{p}\mathrm{d}x.
\end{aligned}
\end{equation*}
We compute
\begin{equation*}
\begin{aligned}
f_{1}'(t)
=&t\|u\|_{E}^{2}
-t^{p-1}
\int_{\mathbb{R}^{N}}|u|^{p}\mathrm{d}x.
\end{aligned}
\end{equation*}
We know that $f_{1}'(\cdot)=0$ iff
\begin{equation*}
\begin{aligned}
\|u\|_{E}^{2}
=t^{p-2}
\int_{\mathbb{R}^{N}}|u|^{p}\mathrm{d}x.
\end{aligned}
\end{equation*}
Let
\begin{equation*}
\begin{aligned}
f_{2}(t)
=t^{p-2}
\int_{\mathbb{R}^{N}}|u|^{p}\mathrm{d}x.
\end{aligned}
\end{equation*}
Clearly, $\lim\limits_{t\rightarrow0}f_{2}(t)\rightarrow0$, $\lim\limits_{t\rightarrow+\infty}f_{2}(t)\rightarrow+\infty$ .
By using the intermediate value theorem, we know that there exists a $0<t_{u}<\infty$ such that
\begin{equation*}
\begin{aligned}
f_{2}(t_{u})=\|u\|_{E}^{2}.
\end{aligned}
\end{equation*}
Furthermore, it is easy to see that $f_{2}(\cdot)$ is strictly increasing on $(0,\infty)$.
Then we get the uniqueness of $t_{u}$.
And then,
\begin{equation*}
\begin{aligned}
\|u\|_{E}^{2}
=t_{u}^{p-2}
\int_{\mathbb{R}^{N}}|u|^{p}\mathrm{d}x,
\end{aligned}
\end{equation*}
which gives
\begin{equation*}
\begin{aligned}
\|t_{u}u\|_{E}^{2}
=\beta\int_{\mathbb{R}^{N}}|t_{u}u|^{p}\mathrm{d}x
+\int_{\mathbb{R}^{N}}|t_{u}u|^{2^{*}}\mathrm{d}x.
\end{aligned}
\end{equation*}
This implies that $t_{u}u\in \mathcal{N}$.

(3)
By applying $\langle I'(u),u\rangle=0$, we know
\begin{equation*}
\begin{aligned}
0
=\langle I'(u),u\rangle\geqslant \|u\|_{E}^{2}
-C\|u\|_{E}^{p},
\end{aligned}
\end{equation*}
which implies
\begin{equation*}
\begin{aligned}
C\|u\|_{E}^{p-2}
\geqslant
1,
\end{aligned}
\end{equation*}
and
\begin{equation*}
\begin{aligned}
\|u\|_{E}^{2}
\geqslant C.
\end{aligned}
\end{equation*}
Then, for $u\in \mathcal{N}$, we get
\begin{equation*}\label{36}
\begin{aligned}
I(u)
=&
I(u)-\frac{1}{p}\langle I'(u),u\rangle\\
=&\frac{1}{2}\|u\|_{E}^{2}
-\beta\frac{1}{p}\int_{\mathbb{R}^{N}}|u|^{p}\mathrm{d}x
-\frac{1}{p}
\left(
\|u\|_{E}^{2}
-\beta\int_{\mathbb{R}^{N}}|u|^{p}\mathrm{d}x
\right)\\
=&
\left(\frac{1}{2}-\frac{1}{ p}\right)
\|u\|_{E}^{2}
\geqslant
C.
\end{aligned}
\end{equation*}
Hence, $I$ is bounded from below on $\mathcal{N}$. And then $\bar{c}>0$.

(4)
By virtue of Lemma \ref{Lemma4.1} (2), we have the following result directly
\begin{equation*}
\begin{aligned}
\bar{c}=\bar{\bar{c}}.
\end{aligned}
\end{equation*}
For any $u\in E\setminus\{0\}$,
there exists some $\tilde{t}>0$ large,
such that $I(\tilde{t}u)<0$.
Define a path $\gamma:[0, 1]\to E$ by $\gamma(t) = t\tilde{t}u$.
Clearly, $\gamma\in\Gamma$ and
\begin{equation*}
\begin{aligned}
c\leqslant \bar{\bar{c}}.
\end{aligned}
\end{equation*}
On the other hand, for each path $\gamma\in\Gamma$,
let $g(t):=\langle I'(\gamma(t)),\gamma(t)\rangle$.
Then, $g(0)=0$ and $g(t)>0$ for $t>0$ small. A direct calculation gives
\begin{equation*}
\begin{aligned}
I(\gamma(1))
-\frac{1}{p}\langle I'(\gamma(1)),\gamma(1)\rangle
\geqslant\left(\frac{1}{2}-\frac{1}{p}\right)\|\gamma(1)\|_{E}^{2}
\geqslant 0,
\end{aligned}
\end{equation*}
which implies
\begin{equation*}
\begin{aligned}
\langle I'(\gamma(1)),\gamma(1)\rangle
\leqslant p\cdot I(\gamma(1))
=p\cdot I(\tilde{t}u)<0.
\end{aligned}
\end{equation*}
Thus, there exists $\tilde{\tilde{t}} \in(0,1)$ such that $g(\tilde{\tilde{t}}) = 0$, i.e. $\gamma(\tilde{\tilde{t}})\in \mathcal{N}$ and $c\geqslant \bar{c}$.
This deduces
$c=\bar{c}=\bar{\bar{c}}$.

(5)
For $u\in \mathcal{N}$, it follows from \eqref{4.1} and \eqref{4.2} that
\begin{equation*}
\begin{aligned}
\langle\Phi'(u),u\rangle
=&\langle\Phi'(u),u\rangle-p\Phi(u)\\
=&\left(2\|u\|_{E}^{2}-p\int_{\mathbb{R}^{N}}|u|^{p}\mathrm{d}x\right)
-p\left(\|u\|_{E}^{2}
-\int_{\mathbb{R}^{N}}|u|^{p}\mathrm{d}x\right)\\
=&(2-p)\|u\|_{E}^{2}
<0.
\end{aligned}
\end{equation*}
Thus, $\Phi'(u)\not=0$ for $u\in \mathcal{N}$.

If $u\in \mathcal{N}$ and $I(u)=\bar{c}$,
from $\bar{c}$ is the minimum of $I$ on $\mathcal{N}$,
and then using the Lagrange multiplier theorem, we know that there exists $\lambda\in \mathbb{R}$ such that $I'(u)=\lambda \Phi'(u)$. So
\begin{equation*}
\begin{aligned}
\langle \lambda \Phi'(u),u\rangle =\langle I'(u),u\rangle=\Phi(u)=0.
\end{aligned}
\end{equation*}
This shows $\lambda=0$ and $I'(u)=0$. Thus, $u$ is a ground state solution for equation \eqref{P}.
\end{proof}

\section*{Acknowledgment}
Yu Su is supported by the Natural Science Research Project of Anhui Educational Committee (2023AH040155).

\section*{Data Availability}
Data sharing is not applicable to this article as no datasets were generated or analysed during the current study.

\section*{Conflict of interest}
On behalf of all authors,
the corresponding author states that there is no conflict of interest.

\end{document}